\documentclass[11pt]{amsart}

\usepackage{hyperref}
\usepackage{mathtools}
\usepackage{amsthm}
\usepackage{amsfonts}
\usepackage[utf8]{inputenc}
\usepackage{float}

\usepackage{tikz}

\numberwithin{equation}{section}
\newtheorem{theorem}{Theorem}[section]

\newtheorem{proposition}[theorem]{Proposition}

\newtheorem{lemma}[theorem]{Lemma}
\newtheorem{conjecture}[theorem]{Conjecture}

\theoremstyle{definition}
\newtheorem{definition}[theorem]{Definition}
\theoremstyle{remark}
\newtheorem{remark}[theorem]{Remark}
\newtheorem{example}[theorem]{Example}

\newcommand{\mC}{\mathcal C}

\newcommand{\mD}{\mathcal D}
\newcommand{\la}{\langle}
\newcommand{\ra}{\rangle}

\newcommand{\C}{\mathbb{C}}

\DeclareMathOperator{\area}{\mathrm{Area}}

\DeclareMathOperator{\vol}{Vol}
\DeclareMathOperator{\Hom}{Hom}

\newcommand{\del}{\partial}
\newcommand{\RR}{\mathbb{R}}
\newcommand{\CC}{\mathbb{C}}
\newcommand{\CP}{\mathbb{CP}}
\newcommand{\HP}{\mathbb{HP}}
\newcommand{\ZZ}{\mathbb{Z}}
\newcommand{\bz}{{\bar{z}}}
\newcommand{\Spin}{\mathrm{Spin}}
\newcommand{\D}{\mathcal{D}}
\newcommand{\bi}[1]{\overline{#1}}

\begin{document}
\title[Dirac eigenvalue optimisation]{Dirac eigenvalue optimisation and harmonic maps to 
complex projective spaces}

\author[M. Karpukhin]{Mikhail Karpukhin}
\address[Mikhail Karpukhin]{Department of Mathematics, University College London, 
Gower Street, London WC1E 6BT, UK}
\email{m.karpukhin@ucl.ac.uk}

\author[A. M\'etras]{Antoine M\'etras}
\address[Antoine M\'etras]{School of Mathematics, University of Bristol,
Fry Building, Woodland Road,
Bristol BS8 1UG,
UK}
\email{antoine.metras@bristol.ac.uk}
\author[I. Polterovich]{Iosif Polterovich
}
\address[Iosif Polterovich]{D\'epartement de math\'ematiques et de 
statistique, Universit\'e de Montr\'eal, CP 6128 Succ. Centre-Ville, 
Montr\'eal, Qu\'ebec, H3C 3J7, Canada}
\email{iosif.polterovich@umontreal.ca}
\begin{abstract}
Consider a Dirac operator on an oriented compact surface endowed with a Riemannian metric and spin structure. Provided the area and the conformal class are fixed, how small can the $k$-th positive Dirac eigenvalue be?    This problem  mirrors the maximization problem for the eigenvalues of the Laplacian, which is related to the study of harmonic maps into spheres. We uncover the connection between the critical metrics for Dirac eigenvalues and harmonic maps into complex projective spaces. Using this approach we show that for many conformal classes on a torus the first nonzero Dirac eigenvalue is minimised by the flat metric. We also present a new geometric proof of B\"ar's theorem stating that  the first nonzero Dirac eigenvalue on the sphere is minimised by the standard round metric.
\end{abstract}
\maketitle
\section{Introduction and main results}
\subsection{Dirac operator on surfaces}
Dirac operator is a first order differential operator defined on the sections of the   spinor bundle
of an oriented Riemannian manifold $(M,g)$. The construction of a spinor bundle relies on an additional structure on $(M,g)$, called the spin structure.  An oriented manifold with a fixed spin structure is referred to as {\em spin manifold}. In general, not all manifolds admit spin structures, however, in the present article we are concerned with the case $\dim M =2$. All oriented surfaces have spin structures and, furthermore, the spinor bundle and the Dirac operator can be described using the complex structure associated with the orientation and the metric, see~\cite{Atiy71} . This is the point of view that we assume below. 

Let $M$ be a compact orientable surface of genus $\gamma$ with a complex structure, and let  $K = T^{(1,0)}M^*$ be its canonical line bundle (here and further on, we refer to \cite{GrHa} for basic facts and terminology from complex geometry). A {\em spin structure} on $M$ is a holomorphic line bundle $S$ and a holomorphic isomorphism $S \otimes S \cong K$, which makes $S$ a square root of $K$. 
The {\it spinor bundle}  $\Sigma M$ on $M$ is $S \oplus \bar{S}$. 
Two spin structures $S$ and $S'$ are said to be equivalent if there exists a holomorphic bundle isomorphism which commutes with the respective isomorphisms $S \otimes S \cong K$ and $S' \otimes S' \cong K$.  One can show that  there are $2^{2\gamma}$ non-equivalent spin structures on $M$.

A {\em Hermitian metric}  on $M$ is a global section $h$ of $K\otimes \bar K$ satisfying $h_p(\xi,\bar\xi)>0$ for all $\xi\in T_p^{(1,0)}M$. It induces a Riemannian metric $g$ on $M$ by taking the real part $g = \mathrm{Re}(h)$,  and a Hermitian metric on $S$, such that for any section $s$ of $S$ one has $|s|^2 = |s\otimes s|$. In local holomorphic coordinates $z$, if $h = e^{2\omega}dz\otimes d\bar z$, then $g = e^{2\omega}(dx^2 + dy^2)$ and if $s_0$ is a local section of $S$ such that $s_0\otimes s_0 = dz$, then $|s_0|= e^{-\omega}$. Viewing a Hermitian metric on $S$ as a map $S\otimes \bar S\to \underline{\C}$, where $\underline{\C}$ is a trivial line bundle, one defines
\begin{align*}
    \bar\del_g : \Gamma(S) \overset{\bar{\del}}{\to} \Gamma(S \otimes \bar{K}) \to \Gamma(\bar{S}),
\end{align*}
where  the second map corresponds to the bundle isomorphisms $S \otimes \bar{K} \cong S \otimes \bar{S} \otimes \bar{S} \cong |S|^2 \otimes \bar{S} \cong_g \bar{S}$. In local coordinates, 
\begin{align*}
    \bar{\del}_g (fs_0) = |s_0|^2_g \del_\bz f \bar{s}_0,
\end{align*}
 for a locally defined complex-valued function $f$.
Similarly, one can define $\del_g : \Gamma(\bar{S}) \to \Gamma(S)$ from the corresponding $\del$-operator.

The Dirac operator $\D_g$ is defined as
\begin{align*}
    \D_g : S \oplus \bar{S} &\to S \oplus \bar{S} \\
    \begin{pmatrix} \psi_+ \\ \psi_- \end{pmatrix} &\mapsto 2 \begin{pmatrix}
        0 & \del_g \\
        - \bar{\del}_g & 0 
    \end{pmatrix}
    \begin{pmatrix} \psi_+ \\ \psi_- \end{pmatrix}.
\end{align*}
The Dirac operator is elliptic and  essentially self-adjoint provided the manifold is complete; moreover, since $M$ is compact, it  has a discrete spectrum consisting of real eigenvalues (see, for instance, \cite[Theorem 1.3.7]{Ginoux}). An easy calculation shows that if $\psi = (\psi_+, \psi_-)$ is a Dirac eigenspinor with eigenvalue $\lambda$, then $(\psi_+, -\psi_-)$ is also a Dirac eigenspinor with the eigenvalue $-\lambda$. Thus, Dirac eigenvalues are symmetric with respect to $0$. Furthermore, $(\bar\psi_-,-\bar\psi_+)$ is an eigenspinor with the eigenvalue $\lambda$. In fact, the map $(\psi_+, \psi_-)\mapsto (\bar\psi_-,-\bar\psi_+)$ commutes with $\D_g$ and can be seen as a {\em quaternionic} structure on the spinor bundle, see~\cite[Section 5]{FLPP} and Remark \ref{quaternionic}.  We use the following notation for the spectrum of $\mD_g$:
\begin{align*}
    -\infty  \cdots \leq \lambda_{-2} \leq \lambda_{-1} < 0 < \lambda_1 \leq \lambda_2 \leq \cdots +\infty,
\end{align*}
where $\lambda_k = - \lambda_{-k}$ and the eigenvalues are counted with their {\em quaternionic multiplicity}, i.e. if $\lambda_1$-eigenspace is spanned by $(\psi_+, \psi_-)$ and $(\bar\psi_-,-\bar\psi_+)$, then $\lambda_2>\lambda_1$.
If the kernel  of $\mD_g$ is non-empty, we do not enumerate the zero eigenvalues.
Finally, we remark that the (quaternionic) dimension $d$ of the kernel of  $\D$ depends {\em only} on the spin structure and the  conformal class. For example,  on the sphere $d=0$, and on the torus $d = 1$  for the trivial spin structure, while $d= 0$ for the other three spin structures.   In general, $d \leq \left[ \frac{\gamma + 1}{2}\right]$ on a surface of genus $\gamma$ \cite{Hit}.

\subsection{Optimisation of Dirac eigenvalues} Let $\mathcal{C}$ be a given conformal class on the surface $M$. To obtain a quantity invariant under  rescaling of the metric, we normalise the Dirac eigenvalues by the area of $M$:
\begin{align*}
    \bar{\lambda}_k(M, g, S) := \lambda_k(M, g, S) \area(M,g)^{1/2}.
\end{align*}
It is known that for all metrics $g \in \mathcal{C}$, the normalised eigenvalue $\bar{\lambda}_k(M, g, S)$ is bounded away from zero \cite{Lott, Am2}. In particular, if $M$ is diffeomorphic to the sphere $\mathbb{S}^2$, then there is a unique spin structure $S$ on $M$,  and Bär's inequality~\cite{B0} states that for all metrics $g$ on $\mathbb{S}^2$,
\begin{align}
\label{ineq:Bar}
    \bar\lambda_1(\mathbb{S}^2,g,S) \geq \bar\lambda_1(\mathbb{S}^2, g_{\mathbb{S}^2},S) = 2 \sqrt{\pi}.
\end{align}
The equality holds if and only if $g$ is a constant curvature metric on $\mathbb{S}^2$, i.e. iff $g$ is homothetic to the standard round metric $g_{\mathbb{S}^2}$ on the unit sphere in $\RR^3$. B\"ar's inequality could be seen as a Dirac analogue of the celebrated Hersch's inequality~\cite{Hersch} for the first eigenvalue of the Laplacian, although their original proofs differ substantially. In section \ref{subsec:energy} we present a new  proof of~\eqref{ineq:Bar} using geometric methods inspired by recent results on optimisation of Laplace eigenvalues, cf. \cite{KS}. 

Finding an explicit value for 
\begin{align*}
    \Lambda_1(M,\mathcal{C}, S) := \inf_{g \in \mathcal{C}} \lambda_1(M, g, S) \area(M,g)^{1/2}
\end{align*}
on other surfaces is an open question. The study of {\it $\bar\lambda_1$-conformally minimal metrics} $g$, i.e. metrics achieving $\Lambda_1(M,\mathcal{C}, S)$, and in particular the question of their existence has been initiated in \cite{Am}. 

\begin{theorem}[\cite{Am}] \label{theo:Am}
    Let $M$ be a compact orientable surface of genus $\gamma$, $S$ a spin structure on it and $\mathcal{C}$ a conformal class of metrics. If
    \begin{align*}
        \inf_{g \in \mathcal{C}} \lambda_1(M,g,S) \area(M,g)^{1/2} < 2 \sqrt{\pi},
    \end{align*}
    then there exists $g_{\min} \in \mathcal{C}$ achieving the infimum. Moreover, the metric $g_{\min}$  is smooth outside of possibly at most $\gamma - 1$ conical singularities. 
\end{theorem}

Our goal is to study {\em critical metrics} for the functionals $\bar\lambda_k(M,g,S)$ for all $k>0$, both in the conformal class and in the space of all metrics. Recall that for Laplace and Steklov eigenvalues, the criticality conditions yield connections to the theory of harmonic maps and minimal surfaces in the sphere and the unit ball, respectively~\cite{ElIl07, FS3}. In fact, we observe that, similarly to those problems, critical metrics for Dirac eigenvalues correspond to a certain class of harmonic maps and minimal surfaces, but the ambient spaces are now {\em complex projective spaces} $\CP^{n}$.

\subsection{Critical metrics and harmonic maps} 
Note that the normalised eigenvalue $\bar\lambda_k$ does not necessarily vary smoothly under smooth variations of the metric, since different branches of $\bar\lambda_k$ intersect at points of multiplicity. The now standard techniques, see~\cite{Nad96,ElIl07, FS3, KM}, allow
one to circumvent this difficulty, properly define critical metrics and obtain the corresponding Euler-Lagrange equations. 
In the case of  $\bar\lambda_k$-conformally critical metrics, we obtain the following characterisation.


\begin{proposition} \label{prop:conf-crit}
    Let $M$ be a compact oriented surface with a fixed conformal class $\mathcal{C}$ and the spin structure $S$. Suppose that $g \in \mathcal{C}$ is a critical metric for $\bar{\lambda}_k(M,g, S)$ in its conformal class. Then there exist eigenspinors $\psi_1, \dots, \psi_m \in E_k(g)$ such that
    \begin{align}
    \label{eq:sum_squares}
        \sum_{j = 1}^m |\psi_j|^2 = 1 \qquad \text{on } M.
    \end{align}
    Conversely, if $\psi_1, \dots, \psi_m \in E_k(g)$ are eigenspinors satisfying~\eqref{eq:sum_squares} and $\lambda_k < \lambda_{k+1}$ or $\lambda_k > \lambda_{k-1}$, then $g$ is conformally critical for $\bar\lambda_k$.
\end{proposition}
\begin{remark} 
\label{rmk:critBG}
The Euler-Lagrange equations for the eigenvalues  of the Dirac operator have been first obtained by Bourguignon and Gauduchon in~\cite{BG}. However,
the definition of criticality used in that paper is more restrictive compared to ours. It coincides with ours if the eigenvalue is simple, and in this case \cite[Proposition 27]{BG} yields the same criticality condition as Proposition~\ref{prop:conf-crit}.
The main difference is that with our definition all minimisers for $\bar\lambda_k$ (including the local ones) are necessarily $\bar\lambda_k$-critical, even if the corresponding eigenvalues are multiple. As a result, we obtain many more examples of critical metrics with multiple eigenvalues,
see Section \ref{sec:qharm}.
\end{remark}
\begin{remark} Note that the extremality condition \eqref{eq:sum_squares} is analogous to the one arising in the case of the Laplacian \cite{Nad96}. The reason is that the Dirac operator on surfaces behaves in a similar manner as the Laplacian 
under a conformal change of the metric: $\D_{e^{2\omega}}g = e^{-\omega}\D_g$, i.e. the operator is multiplied by a function on the left.  Another manifestation of this phenomenon can be found in \cite[Theorem 1.1]{ PA22}.
\end{remark}
\begin{remark} 
\label{rmk:lambda1_min}
    If we consider a locally minimal metric for $\lambda_1$, then we can always take $m = 1$ in the criticality condition~\eqref{eq:sum_squares}. Indeed, for such a  metric, tangent lines to different eigenvalue branches for $\lambda_1$ cannot intersect at one point,  as that would imply that some variations of the metric decrease $\bar \lambda_1$, contradicting the minimality. Hence, all those tangent lines are horizontal and the derivatives of all the branches of $\lambda_1$ evaluated at such a metric must vanish. The exact expression for the derivative obtained in~\eqref{eq:eig_deriv} implies that all $\lambda_1$-eigenspinors must have constant length.
\end{remark}

 A statement analogous to Proposition~\ref{prop:conf-crit} holds for conformally critical metrics for Laplace eigenvalues. In that case,  condition~\eqref{eq:sum_squares} can be phrased as existence of a map to a unit sphere by eigenfunctions with the same eigenvalue, which turns out to be equivalent to the  harmonicity of the map. Since eigenspinors are sections of non-trivial bundles, they do not form a map to a sphere, but rather to a projective space (see \cite[Chapter 1.4]{GrHa}).
 
  Let $\psi_1, \dots, \psi_m$ be $\D_g$-eigenspinors such that  $\psi_j = (\psi_{j+}, \psi_{j-})$ where $\psi_{j+} \in \Gamma(S), \psi_{j-} \in \Gamma(\bar{S})$. It follows  that $(\psi_{j+}, \bar{\psi}_{j-})$ is a section of $S \oplus S$. If $Z = \{p \in M \,|\, \psi_1(p) = \dots = \psi_m(p) = 0\}$,  then we can define a map $\Psi\colon M \setminus Z \to \CP^{2m-1}$ given by
\begin{equation}
\label{eq:mapcp}
\Psi = [\psi_{1+} : \bar\psi_{1-}: \dots : \psi_{m+}: \bar\psi_{m-}]. 
\end{equation}
We show that the map $\Psi$ extends continuously to the whole surface $M$, see Lemma \ref{lem:extend}.
One of the main observations of the present paper is that if the eigenspinors satisfy the criticality condition~\eqref{eq:sum_squares}, then the map $\Psi$ is harmonic. This should be compared  to the  analogous result for the Laplacian~\cite{Nad96, ElIl03}.

\begin{theorem} \label{prop:harmo}
    Let $(M,g)$ be  a compact oriented surface with spin structure $S$.  Suppose that $\psi_1, \dots, \psi_m$ are eigenspinors on $M$ with $\D_g \psi_j = \lambda \psi_j, \lambda \neq 0$ and $|\psi_1|_g^2 + \ldots + |\psi_m|_g^2 = 1$ on $M$. Then the map $\Psi\colon M \to \CP^{2m-1}$ given in homogeneous coordinates by $[\psi_{1+} : \bar{\psi}_{1-} : \ldots : \psi_{m+} : \bar{\psi}_{m-}]$ is harmonic. 
\end{theorem}

Observe that not all harmonic maps to $\CP^{2m-1}$ have the form~\eqref{eq:mapcp}. In Definition~\ref{def:spinharm} we identify a class of harmonic maps called {\it quaternionic harmonic}, which turn out to be exactly the maps formed by Dirac eigenspinors satisfying~\eqref{eq:sum_squares}. As a result, there is a two-way correspondence between conformally critical metrics and quaternionic harmonic maps to $\CP^{2m-1}$ for some $m>0$. For the precise statement we refer to Proposition~\ref{prop:spin-harmonic} below.

Finally, we consider criticality condition in the space of all Riemannian metrics and spin structures. It turns out that for a metric to be critical in this larger space, the quaternionic harmonic map $\Psi$ needs to be additionally  conformal, see Section~\ref{sec:glob_crit} for precise statements.
\subsection{Conformal minimisers on  the torus}
It has been observed in~\cite{KS} that the conformal optimisation of the Laplace eigenvalues is closely connected to the min-max theory for the energy functional on the space of sphere-valued maps. In view of Theorem~\ref{prop:harmo} and Proposition~\ref{prop:spin-harmonic}, it is natural to expect that conformal minimisation of Dirac eigenvalues is related to the min-max theory for the energy functional on the space of maps to projective spaces. In Section~\ref{sec:energyproof} we present first results in this direction yielding two applications: we obtain a geometric proof of B\"ar's inequality~\eqref{ineq:Bar} and, more importantly, we show that the flat metric on a torus is $\bar\lambda_1$-conformally minimising on many conformal classes for either spin structure. Below we  state our results for the {\em trivial} spin structure on the torus; see Section \ref{subsec:torus} for analogous statements for the non-trivial spin structures.

For any vector $(a,b)\in \mathbb{R}^2$, $b\ne 0$, one has the lattice $\Gamma=\Gamma_{(a,b)} = \mathbb{Z}(1,0) + \mathbb{Z}(a,b)$ and the corresponding flat torus $\mathbb{T}^2 = \mathbb{R}^2/\Gamma_{a,b}$ with the flat metric $g_{a,b}$ induced by the projection map.
 It is well-known that any conformal class contains a unique flat metric of unit area, and the moduli space of conformal classes is parametrised by the points $(a,b)$ such that $0\le a\le 1/2$ and $a^2+b^2 \ge 1$. Let $[g_{a,b}]$ be the conformal class containing the metric $g_{a,b}$.

The trivial spin structure $S_0$ on a torus is a trivial line bundle,  and the spinors can be simply viewed as pairs of complex functions. An elementary computation shows that
\[
\bar\lambda_1(\mathbb{T}^2,g_{a,b}, S_0) = \frac{2\pi}{\sqrt{b}}.
\]
The following theorem is one of the main results of the paper.
\begin{theorem} \label{thm:torustrivial}
   For  all $0\le a \le 1/2$ and $b > 2\pi$, one has 
    $$
    \Lambda_1(\mathbb{T}^2,[g_{a,b}],S_0) = \bar\lambda_1(\mathbb{T}^2,g_{a,b}, S_0) = \frac{2\pi}{\sqrt{b}}.
    $$
    Furthermore, the flat metric is the unique smooth minimiser.
\end{theorem}
The proof of Theorem \ref{thm:torustrivial} is given in Section \ref{subsec:torus}; in fact, we prove a more general Theorem \ref{thm:torus} which also applies to non-trivial spin structures. 
It follows from the argument that for the trivial spin structure, the only natural smooth candidate for the minimiser is the flat metric. Together with the existence result of Theorem \ref{theo:Am}, this leads us to the  following conjecture.
\begin{conjecture}
\label{conj:torus}
For the trivial spin structure $S_0$ on the torus $\mathbb{T}^2$ one has
\begin{equation*}
\Lambda_1(\mathbb{T}^2,[g_{a,b}],S_0) = 
\begin{cases}
\frac{2\pi}{\sqrt{b}} & \text{ if $b\geq \pi$};\\
2\sqrt{\pi} & \text{ if $b\leq\pi$}.
\end{cases}
\end{equation*}
\end{conjecture}
In other words, we conjecture that either the minimiser is flat, or there is no smooth minimiser and the minimising sequence degenerates to a ``bubble''. If  true, this contrasts with the case of the Laplacian, for which bubbling can not occur for conformally maximal metrics for the first eigenvalue on surfaces of positive genus \cite{Pet14}. 
\begin{remark}  Theorem \ref{thm:torustrivial} and, more generally, Theorem \ref{thm:torus}, can be compared with the similar statements for the first Laplace eigenvalue. In that case, flat metrics are only known to be maximisers if $a^2+b^2=1$ \cite{EIR}. Moreover, as follows from \cite{FN, Pet14}, flat metrics  can not be maximisers for $b \ge \pi/2$.
\end{remark}

\subsection{Plan of the paper}
We start by proving Proposition~\ref{prop:conf-crit} in Section~\ref{sec:proofcrit}, which is done  similarly  to the  analogous results for Laplace and Steklov eigenvalues~\cite{Nad96, ElIl07, KNPP, FS3}. Section~\ref{sec:harm} contains the necessary background material on harmonic maps to $\CP^n$, which allows us to prove Theorem~\ref{prop:harmo} in Section~\ref{sec:geomcrit}. We then investigate additional properties satisfied by the maps given by eigenspinors~\eqref{eq:mapcp},  leading us to the notion of {\em quaternionic harmonic maps}, see Definition \ref{def:spinharm}. In particular, Section~\ref{sec:eigmaps} contains Proposition~\ref{prop:H_harm} stating that any quaternionic harmonic map $\Psi\colon M\to\CP^{2m-1}$ induces a spin structure on $M$ such that the components of $\Psi$ are Dirac eigenspinors. This yields  the converse to Theorem~\ref{prop:harmo}, see Proposition~\ref{prop:spin-harmonic}.  Some explicit examples of quaternionic harmonic maps are presented in Section~\ref{sec:qharm}.

In Section~\ref{sec:energyproof},  we specialise to the first non-zero Dirac eigenvalue. For a conformal class $\mC$ and a spin structure $S$ on $M$, Theorem~\ref{thm:energy} gives a characterisation of $\Lambda_1(M,\mC,S)$ in terms of the energy of quaternionic harmonic maps. Recall that by Remark~\ref{rmk:lambda1_min},  the minimisers for $\lambda_1$ yield a map to $\CP^1$. Furthermore, Definition~\ref{def:spinharm} of a quaternionic harmonic map to $\CP^1$ essentially reduces to non-holomorphicity (up to a condition on branch points). Thus, informally, Theorem~\ref{thm:energy} states that the conformal minimum of the first normalised Dirac eigenvalue equals to the square root of the lowest possible energy of a non-holomorphic harmonic map to $\CP^1$. B\"ar's inequality then follows from a well-known fact that for $M=\mathbb{S}^2$, the lowest energy of a non-constant harmonic map $M\to\CP^1=\mathbb{S}^2$ is $4\pi$ and can be achieved on an anti-holomorphic map given by  reflection across the equator~\cite{C}. Our main application of Theorem~\ref{thm:energy} is Theorem~\ref{thm:torus}, which is an extension of Theorem \ref{thm:torustrivial}. It characterises flat metrics as unique $\bar\lambda_1$-conformal minimisers for many conformal classes on the torus.
The proof uses similar ideas as in the sphere case, with the main new ingredient being~\cite[Corollary 6.6]{FLPP},  which states that any harmonic map $\mathbb{T}^2\to\mathbb{S}^2$ of energy below $4\pi$ is a map to the equator $\mathbb{S}^1\subset\mathbb{S}^2$. 

Finally, in Section~\ref{sec:glob_crit} we investigate globally critical metrics for Dirac eigenvalues. Our main result is Proposition~\ref{prop:globcrit}, which states that globally critical metrics correspond to  quaternionic harmonic maps to $\CP^{2m-1}$ which are branched minimal immersions.

\section{Conformally critical metrics and harmonic maps}

\subsection{Criticality condition in the conformal class}
\label{sec:proofcrit}

The following definition is a Dirac analogue of the notion of critical metrics for Laplace eigenvalues, see \cite{Nad96, ElIl03, ElIl07} (note that the term ``extremal" is sometimes used instead of ``critical"). 
For an analytic variation $g(t)$ of $g$, the eigenvalues form analytic branches, see e.g.~\cite{BG} for an explanation of this in the context of Dirac eigenvalues. However, for multiple eigenvalues, the multiplicity is usually destroyed by the perturbation and the {\em enumerated} eigenvalues $\lambda_k$ are not analytic due to the fact that different analytic branches correspond to the $k$-th eigenvalue for $t>0$ and $t<0$. Nevertheless, the right and left derivatives of $\lambda_k(g(t))$ exist and can be defined in the following way: if $l$ is the multiplicity of $\lambda_k(g)$, then there exist analytic families of real numbers $\lambda^{(1)}(t), \dots, \lambda^{(l)}(t) \in \mathbb{R}$ and spinors $\phi^{(1)}(t), \dots, \phi^{(l)}(t)$ orthonormal in $L^2(g(t))$, such that $\D_{g(t)} \phi^{(j)}(t) = \lambda^{(j)} \phi^{(j)}(t)$ for $j = 1, \dots, l$. Then the left derivative of $\lambda_k(g(t))$ at $g$ is
\begin{align*}
\frac{d}{dt} \lambda_k(g(t))\Big|_{t = 0^-} = \frac{d}{dt}\lambda^{(j)}(t)\Big|_{t = 0}
\end{align*}
for some $j$ such that $\lambda_k(g(t)) = \lambda^{(j)}(t)$ for $t < 0$. The right derivative is defined similarly. 

The existence of  the left and right derivatives makes the following definition possible.
\begin{definition}
Let $M$ be an oriented surface endowed with a conformal class $\mC$ and a spin structure $S$. We say that $g\in\mC$ is {\em $\bar\lambda_k$-conformally critical} if for any analytic family of smooth metrics $g(t)\in \mC$, $g(0) = g$, $t\in(-\varepsilon,\varepsilon)$,  one has 
    \begin{align*}
        \left(\frac{d}{dt}\bar\lambda_k(M,g(t), S)\Big|_{t = 0^-} \right) \left(\frac{d}{dt}\bar\lambda_k(M,g(t), S)\Big|_{t = 0^+}\right) \leq 0.
    \end{align*}
\end{definition}

\begin{proof}[Proof of Proposition~\ref{prop:conf-crit}]
    Given an analytic one-parameter family of metrics $g(t) = e^{2\omega(t)}g, t \in (-\epsilon, \epsilon)$ with $\omega(0) = 0$, let $\psi(t)$ be an analytic family of eigenspinors associated to the analytic eigenvalue branch $\lambda^{(j)}(t)$, $\lambda^{(j)}(0) = \lambda_k(g)$. Then the derivative of $\lambda^{(j)}(t)$ is 
    \begin{align} \label{eq:eig_deriv}
        \frac{d \lambda^{(j)}}{dt}(t) = - \lambda_k(g) \int_M \dot\omega(t) |\psi|^2_{g(t)} dv_{g(t)}.
    \end{align}
    Formula~\eqref{eq:eig_deriv} was obtained in~\cite{BG}, but we provide a proof here for the sake of completeness. Taking the derivative with respect to $t$ of the eigenvalue equation $\D_{g(t)} \psi(t) = \lambda^{(j)}(t) \psi(t)$ and using that $\D_{g(t)} = e^{-\omega(t)} \D_{g(0)}$ yields 
    \begin{align*}
        0 = - \dot\omega(t) D_{g(t)} \psi(t) + D_{g(t)} \dot\psi(t) - \dot\lambda^{(j)}(t) \psi(t) - \lambda^{(j)}(t) \dot\psi(t).
    \end{align*}
    We then take the Hermitian product with $\psi(t)$ and integrate over $(M,g(t))$, using that $D_{g(t)}$ is self-adjoint:
    \begin{align*}
        0 &= \int_M - \dot\omega(t) \lambda^{(j)}(t) |\psi(t)|^2_{g(t)} + \langle \dot\psi, \D_{g(t)} \psi \rangle - \dot\lambda^{(j)}(t) |\psi(t)|^2_{g(t)} - \lambda^{(j)}(t) \langle \dot\psi, \psi \rangle dv_{g(t)} \\
          &= - \lambda^{(j)}(t) \int_M \dot\omega(t) |\psi|^2 dv_{g(t)} - \dot\lambda^{(j)}(t) \int_M |\psi|^2 dv_{g(t)},
    \end{align*}
    which gives the desired formula.

         We now show that if $g$ is $\bar\lambda_k$-conformally critical, then for any $\dot\omega \in C^{\infty}(M)$ with $\int_M \dot\omega dv_g = 0$, there exists $\psi \in E_k(g)$ such that
    \begin{align*}
        \int_M \dot\omega |\psi|^2_g dv_g = 0.
    \end{align*}
    Let $g(t) = e^{\dot\omega t} g$. Since $g$ is conformally critical, without loss of generality we have
    \begin{align*}
        \frac{d\bar\lambda_k}{dt}\Big|_{t = 0^-} \geq 0 \quad \text{and} \quad \frac{d\bar\lambda_k}{dt}\Big|_{t = 0^+} \leq 0.
    \end{align*}
    Since $\int_M \dot\omega dv_g = 0$, this implies the existence of $\phi_1,\phi_2 \in E_k(g)$ such that 
    \begin{align*}
        \int_M \dot\omega |\phi_1|^2_g dv_g \geq 0 \quad \text{and} \quad \int_M \dot\omega |\phi_2|^2_g dv_g \leq 0,
    \end{align*}
       and taking a linear combination of these two spinors yields the desired $\psi$. 

    Finally, we show that one can take a collection $\psi_1, \dots, \psi_m$ of $\lambda_k(g)$-eigenspinors such that $\sum |\psi_j|^2 = 1$. Let 
    \begin{align*}
        \mathcal{W} = \mathrm{conv} \{ |\psi|^2 , \psi \in E_k(g)\} \subset C^\infty(M)
    \end{align*}
    be the convex hull of the $|\psi|^2$,  and suppose on the contrary that $1 \notin \mathcal{W}$. 
    Since $\mathcal{W}$ is a finite dimensional convex cone, we can apply the Hahn--Banach  theorem to separate 1 from $\mathcal{W}$:  there exists $\xi \in C^\infty(M)$ such that
    \begin{align*}
        &\int_M \xi dv_g > 0, \\
        &\int_M \xi |\psi|^2 dv_g \leq 0 \quad \forall \psi \in E_k(g).
    \end{align*}
    Let $\xi_0 = \xi - \frac{1}{\vol(M,g)} \int_M \xi dv_g$, so that $\int_M \xi_0 = 0$. By the previous result, there exists $\psi \in E_k(g)$ such that
    \begin{align*}
        0 &= \int_M \xi_0 |\psi|^2 dv_g \\
          &= \int_M \xi |\psi|^2 dv_g - \frac{1}{\vol(M,g)} \int_M \xi dv_g \int_M |\psi|^2 dv_g < 0.
    \end{align*}
    a contradiction so $1 \in \mathcal{W}$. 

   We now prove Proposition\ref{prop:conf-crit} in the other  direction. Let $\psi_1, \dots, \psi_m \in E_k(g)$ be such that $1 = \sum |\psi_j|^2$,  and assume without loss of generality that $\lambda_k(g) < \lambda_{k+1}(g)$. Suppose, on the contrary, that $g$ is not $\bar\lambda_k$-conformally critical. Then there exists an analytic family of smooth metrics $g(t)$ with constant volume and $\frac{dg}{dt} = \dot\omega g$, such that the left and right derivatives of $\bar\lambda_k$ at $g(0) = g$ are either both negative or both positive. Without loss of generality, assume they are both negative. 
    There exists a basis $\phi_1, \dots, \phi_l \in E_k(g)$, such that the derivative of any eigenvalue branch $\lambda^{(j)}$ at $g$ is given by $-\int_M \dot\omega |\phi_j|^2 dv_g$ for some $\phi_j$. Since $\lambda_k(g) < \lambda_{k+1}(g)$, the right derivative of $\lambda_k(g(t))$ at $g$ must be the biggest right derivative of all the branches at $g$. So for all $\phi \in E_k(g)$, $-\int_M \dot\omega |\phi|^2 dv_g$ must be negative. Then
    \begin{align*}
        0 = \int_M \dot\omega dv_g = \sum_{j = 1}^m \int_M \dot\omega |\psi_j|^2 dv_g > 0,
    \end{align*}
    and we get a  contradiction.
\end{proof}

\subsection{Harmonic maps to $\CP^n$}

\label{sec:harm}
In this section we review the necessary facts about the geometry of harmonic maps. We are mainly concerned with maps from surfaces to $\CP^n$, so this falls into the setting of harmonic maps between K\"ahler manifolds. We refer to~\cite{EW2} for a more comprehensive exposition.

Given a map  $\Psi\colon (M,g) \to (N, \tilde{g})$, its energy $E(\Psi)$ is given by
\begin{align*}
    E(\Psi) = \frac{1}{2}\int_M |d\Psi|^2 dv_g, 
\end{align*}
where $|\cdot|$ is the Hilbert--Schmidt norm computed with respect to metrics $g,\tilde{g}$. The critical points of the energy functional are called {\em harmonic maps}. 

The complexification $d_{\C}\Psi\colon T_\C M\to T_\C N$ induces the maps $\partial \Psi\colon T^{(1,0)}M\to T^{(1,0)}N$ and $\bar\partial  \Psi \colon T^{(0,1)}M\to T^{(1,0)}N$,  obtained by composing the restriction of $d_{\C}\Psi$ on a corresponding subspace of $T_\C M$ with the projection onto $T^{(1,0)} N$. If manifolds $(M,g)$ and $(N,\tilde{g})$ are Hermitian, then one can readily see that $|d\Psi|^2 = 2(|\partial \Psi|^2 + |\bar\partial \Psi|^2)$, and hence
\begin{align*}
    E(\Psi) = \underbrace{\int_M |\del \Psi|^2 dv_g}_{E^{(1,0)}(\Psi)} + \underbrace{\int_M |\bar\del \Psi|^2 dv_g}_{E^{(0,1)}(\Psi)}.
\end{align*}

In terms of the corresponding K\"ahler forms $\omega_M$, $\omega_N$,  one has that 
\begin{equation}
\label{eq:diff_Khlr}
|\del \Psi|^2 - |\bar\del \Psi|^2 = \la\omega_M,\Psi^*\omega_N\ra.
\end{equation}
Therefore, if $M,N$ are K\"ahler, then the difference $E^{(1,0)}(\Psi) - E^{(0,1)}(\Psi)$ depends only on the homotopy class of the map $\Psi$. As a result, $\Psi$ is harmonic if and only if $\Psi$ is a critical point $E^{(1,0)}$, which is in turn  if and only if $\Psi$ is a critical point of $E^{(0,1)}$.

We now consider the case $N = \CP^n$. The holomorphic tangent bundle $T^{(1,0)}\CP^{n}$ of the complex projective space can be identified with the space $\Hom_\C(L,L^\perp)$ of $\C$-linear maps from the tautological bundle $L$ over $\CP^{n}$ to its orthogonal complement $L^\perp$. Explicitly, let $h\colon T^{(1,0)}\CP^{n} \to \Hom_\C(L, L^\perp)$ be this identification, then for a vector $Z \in T^{(1,0)}_p\CP^{n}$,
\begin{align*}
    h(Z)(l) := \begin{cases}
        0 & \text{if } l = 0 \\
        \pi_{L_p^\perp} \frac{\del f}{\del z}(0) & \text{otherwise},
    \end{cases}
\end{align*}
where $L_p$ is a point $p$ viewed as a line in $\C^{n+1}$, $\pi_{L_p^\perp}$ is the orthogonal projection onto $L_p^{\perp}$, $\phi\colon U\to \CP^n$ be a map from a neighbourhood of $0$ in $\C$ such that $\phi(0) = p$, $\frac{\del\phi}{\del z}(0) = Z$,  and $f\colon U\to \C^{n+1}$ is such that $f(0) = l$ and $f(z) \in \phi(z)$. 

The space $\Hom_\C(L,L^\perp)$ is endowed with a Hermitian metric and a holomorphic connection $\nabla$ induced from a trivial $\C^{n+1}$-bundle over $\CP^{n}$. If we consider $\CP^n$ with its usual complex structure and Fubini-Study metric $g_{FS}$ of constant holomorphic curvature $4$, then $h$ preserves both the metric and a complex structure. For our purposes it is convenient to identify $\CP^1$ with the unit sphere $\mathbb{S}^2$, thus, instead we consider the metric $g_{\CP^n} = 4g_{FS}$, so that this identification is an isometry.

Let $(M,g)$ be a surface 
endowed with a complex structure. Any map $\Psi\colon M \to \CP^{n}$ is given locally in homogeneous coordinates as $\Psi = [F]$ for some nonvanishing function $F\colon M \to \C^{n+1}$. Then for $Z \in T^{(1,0)}M$ one has
\begin{align*}
    h(\del\Psi(Z))(F) = \pi_{L^\perp} d_\C F(Z),\qquad h(\bar\del\Psi(\bar Z))(F) = \pi_{L^\perp} d_\C F(\bar Z).
\end{align*}
 From now on, the identification $h$ is kept implicit. 

It is easy to see that $\Psi$ is critical for $E^{(0,1)}$ if and only if one has 
\begin{equation}
\label{eq:def_harm}
\nabla_{\del/\del z} \bar\del\Psi = 0,
\end{equation}
i.e. if and only if $\bar\del\Psi$ is an anti-holomorphic element of $\Hom_\C(\Psi^*L,\Psi^*L^\perp)$. If $\Psi$ is not holomorphic, i.e. if $\bar \del \Psi\not\equiv 0$, then outside of finitely many points the image of $\bar\del\Psi$ forms a line in $\Psi^*L^\perp$. One can use anti-holomorphicity of $\bar\del\Psi$ to extend this line bundle across zeroes of 
$\bar\del\Psi$.  The resulting line bundle is called $L_{-1}$, while $\Psi^*L$ is referred to as $L_0$.  Let us endow $L_{-1}$ with the holomorphic structure induced from the Hermitian metric on $\C^{n+1}$. Then the map $\bar\del\Psi\colon L_0\to L_{-1}$ becomes anti-holomorphic. To every such map $\Psi$ one can associate a whole sequence of complex line bundles $L_i$ called {\em harmonic sequence}, see~\cite{Wolfson}. In the present paper we only need $L_{-1}$.

With our convention that $g_{\CP^n} = 4g_{FS}$ we have
\begin{equation}
\label{eq:FS}
|\del\Psi|^2 = 4\frac{|\pi_{L^\perp}\partial_zF|^2}{|\del_z|^2|F|^2},\qquad  |\bar\del\Psi|^2 = 4\frac{|\pi_{L^\perp}\partial_{\bar z}F|^2}{|\del_{\bar z}|^2|F|^2},
\end{equation}
where we write $\del_z$ and $\del_{\bar z}$ for $\frac{\del}{\del z}$ and $\frac{\del}{\del \bar z}$,  respectively.
Finally, we note that with our convention,  integrating~\eqref{eq:diff_Khlr} gives (see \cite[p. 247]{EW2})
\begin{equation}
\label{eq:Edeg}
E^{(1,0)}(\Psi) - E^{(0,1)}(\Psi) = 4\pi\deg(\Psi).
\end{equation}
%
%
%
%
%
%

\subsection{Geometric criticality condition in the conformal class}

\label{sec:geomcrit}

\begin{lemma}
\label{lem:extend}
   The map  $\Psi$ defined by \eqref{eq:mapcp} extends continuously to the whole manifold $M$.
\end{lemma}
\begin{proof}[Proof of Lemma \ref{lem:extend}]
    The proof proceeds in two steps: first we investigate the behavior of a single spinor $\psi_j$ near its zeroes, and then consider the projectivisation $[\psi_{1+} : \bar{\psi}_{1-} : \dots ]$.

    Fix an eigenspinor $\psi$,  and let $p \in M$ be a zero of $\psi$, $\psi(p) = 0$. Let $z$ be a local holomorphic coordinate centered at $p$, and $s_0$ be a local section of $S$ with $s_0 \otimes s_0 = dz$.  We write 
    \begin{equation}
    \label{eq:psi}
    \psi = (f_+ s_0, \bar{f}_- \bar{s}_0).
    \end{equation}
    Expanding around $p$, we have
    \begin{align*}
        f_{\pm}(z,\bz) = P_{\pm}(z,\bz) + R_{\pm}(z,\bz),
    \end{align*}
    where $P_\pm$ is a homogeneous polynomial of order $k_\pm$ and $R_\pm$ corresponds to higher order terms. Without loss of generality, we assume $k_- \leq k_+$. 

    From the fact that $\psi$ is an eigenspinor, we have $\del_\bz f_- = \frac{\lambda}{\sqrt{2} |s_0|^2} \bar{f}_+$, i.e.
    \begin{align*}
        \del_\bz P_- + \del_\bz R_- = \frac{\lambda}{2|s_0|^2} (\bar{P}_+ + \bar{R}_+).
    \end{align*}
    The term $\del_\bz P_-$ is of order strictly less than $k_-$ or identically zero, while the terms on the right-hand side are of order $k_+ \geq k_-$. Hence we conclude that $\del_\bz P_- = 0$, so that $P_-(z,\bz) = a z^{k_-}$ for some $a \neq 0$. 

    Looking at $f_+$ and using again that $\psi$ is an eigenspinor, we obtain that 
    \begin{align*}
        P_+(z,\bz) = \sum_{\substack{k' + k'' = k_+  \\ k' \geq k_-}} b_{k',k''} z^{k'}\bz^{k''},
    \end{align*}
    with at least one of the coefficients $b_{k',k''}$ being non-zero. 

    We now consider a family of eigenspinors $\psi_1, \dots, \psi_m$. We again write $\psi_j = (f_{j+} s_0, \bar{f}_{j-} \bar{s_0})$,  and similarly we have the expansion in terms of $P_{j\pm}$ and $R_{j\pm}$. 
    Let $K_j = \min \{k_{j+}, k_{j-}\}$ and $K = \min_{j} K_j$. Then by the previous reasoning, the limits $\lim_{q \to p} \frac{f_{j\pm}}{z(p)^{K}}$ are well defined for all $j\pm$,  and at least one of them is equal to $1$ (to simplify notation, we suppose it is the case for $f_{1-}$). Hence, 
    \begin{align*}
        [\psi_{1+} : \bar{\psi}_{1-} : \dots : \psi_{m+} : \bar{\psi}_{m-}] &= [f_{1+} : f_{1-} : \dots : f_{m+} : f_{m-}] \\
                                                                            &= [f_{1+}/z^K : f_{1-}/z^K : \dots : f_{m+}/z^K : f_{m-}/z^K] \\
                                                                            &= [f_{1+}/z^K : 1 : \dots : f_{m+}/z^K : f_{m-}/z^K] 
    \end{align*}
    is a well defined projective point if $z = 0$, and we can set $\Psi(p)$ equal to it.
\end{proof}

\label{sec:proofharmo}
\begin{proof}[Proof of Theorem \ref{prop:harmo}]
    Let $z$ be some local holomorphic coordinate and $s_0$ be a local section of $S$ such that $s_0 \otimes s_0 = dz$. Write $\psi_j = (f_{j+} s_0, \bar{f}_{j-} \bar{s}_0)$ and $F = (f_{1+}, f_{1-}, f_{2+}, f_{2-}, \dots)$. Locally, $\Psi$ is the projectivisation $[F]$ of $F$. By~\eqref{eq:def_harm}, in order to prove that $\Psi$ is harmonic it is sufficient to show that $(\nabla d\Psi)(\del_z, \del_\bz) = 0$ for any choice of local coordinates.

    As before, let $L \subset \CP^{2m-1} \times \mathbb{C}^{2m}$ be the tautological bundle and $L^\perp$ its orthogonal complement. Then
    \begin{align*}
        d\Psi(\del_\bz)(F) = \pi_{L^\perp} \del_\bz F,
    \end{align*}
    and
    \begin{align}
    \label{eq:defharm1}
        (\nabla d \Psi(\del_\bz, \del_z))(F) = \pi_{L^\perp} \del_z\left( \pi_{L^\perp} \del_\bz F\right) - d\Psi(\del_\bz)(\pi_L \del_z F),
    \end{align}
    where we use $\pi_E$ to denote the orthogonal projection on the subspace $E$.

    From the eigenvalue equation $\D \psi_j = \lambda \psi_j$, we obtain
    \begin{align} \label{eq:eigenvalue}
        2 |s_0|^2 \del_\bz f_{j+} = - \lambda \bar{f}_{j-}, \\
        2 |s_0|^2 \del_z \bar{f}_{j-} = \lambda f_{j+}.
    \end{align}
    As a result,
    \begin{equation}
    \label{eq:eigsp_qharm}
    \del_\bz (f_{j+}, f_{j-}) = \frac{\lambda}{2|s_0|^2} (- \bar{f}_{j-}, \bar{f}_{j+}).
    \end{equation}
    In particular, $\del_\bz F$ is orthogonal to $F$, hence
    \begin{align*}
        d\Psi(\del_\bz)(F) = \pi_{L^\perp} \del_\bz F = \del_\bz F.
    \end{align*}
    Then
    \begin{align*}
        (\nabla d\Psi(\del_\bz, \del_z))(F) = \pi_{L^\perp} \del_z \del_\bz F - d \Psi(\del_\bz)(\pi_{L} \del_z F).
    \end{align*}
    The first term on the right-hand side of~\eqref{eq:defharm1} gives
    \begin{align*}
        \pi_{L^\perp} \left(\del_z \del_\bz F\right) &= \frac{\lambda}{2} \del_z \left(\frac{1}{|s_0|^2}\right) ( -\bar{f}_{1-}, \bar{f}_{1+}, \dots ) + \frac{\lambda}{2 |s_0|^2} \pi_{L^\perp} \del_z (-\bar{f}_{1-}, \bar{f}_{1+}, \dots) \\
                                        &= \frac{\lambda}{2} \del_z \left(\frac{1}{|s_0|^2} \right) (-\bar{f}_{1-}, \bar{f}_{1+}, \dots),
    \end{align*}
    where the eigenvalue equation is used to conclude that $\del_z (-\bar{f}_{1-}, \bar{f}_{1+}, \dots) = - \frac{\lambda}{2|s_0|^2}F$ and is annihilated by $\pi_{L^\perp}$. 

    The second term on the right-hand side of~\eqref{eq:defharm1} gives
    \begin{align*}
        \pi_{L^\perp} \del_\bz \left(\frac{\langle \del_z F, F \rangle}{|F|^2} F \right)
        &= \frac{\langle \del_z F, F \rangle} {|F|^2} \del_\bz F \\
        &= \frac{\lambda}{2|s_0|^2} \frac{\langle \del_z F, F \rangle}{|F|^2} (-\bar{f}_{1-}, \bar{f}_{1+}, \dots).
    \end{align*}

    It then suffices to show that $\del_z \left(\frac{1}{|s_0|^2}\right) = \frac{\langle \del_z F, F \rangle}{|s_0|^2 |F|^2}$. This follows from the condition $|\psi_1|^2 + \dots + |\psi_m|^2 = 1$. Indeed, 
    \begin{align*}
        1 = |\psi_1|^2 + \dots + |\psi_m|^2 = |F|^2 |s_0|^2
    \end{align*}
    so $\del_z |F|^2 = \del_z \left(\frac{1}{|s_0|^2}\right)$. But $\del_z |F|^2 = \langle \del_z F, F \rangle + \langle F, \del_\bz F \rangle = \langle \del_z F, F \rangle$.
\end{proof}

We can also observe that the metric $g$ and the value of $\bar\lambda_k(M,g,S)$ can be computed in terms of the map $\Psi$. Indeed, by~\eqref{eq:eigsp_qharm} one has
$$
|\bar\del\Psi|^2_g = 4\frac{|\del_{\bar z}F|^2}{|F^2||\del_z|_g^2} = \lambda_k^2(M,g,S),
$$
so that $g$ is proportional to $g_\Psi = |\bar\del\Psi|_g^2g$. In particular, $\lambda_k(M,g_\Psi, S)=1$.
Furthermore, 
$$
\bar\lambda_k(M,g,S)^2 = \int_M|\bar\del\Psi|^2_g\,dv_g = E_g^{(0,1)}(\Psi).
$$

\subsection{Maps by eigenspinors}
\label{sec:eigmaps}
Consider the inverse problem: when can a harmonic map to $\CP^{2m-1}$ be written in terms of eigenspinors? 
We first investigate what conditions are satisfied by such a map.
Let $\Psi = [\psi_{1+} : \bar\psi_{1-} : \dots : \psi_{m+} : \bar\psi_{m-}]$, where $\psi_j = (\psi_{j+}, \psi_{j-})$ are eigenspinors $\D \psi_j = \lambda \psi_j$ satisfying $|\psi_1|^2 + \dots + |\psi_m|^2 = 1$.
Consider $I\colon \CC^{2m} \to \CC^{2m}$ given by
\begin{align*}
    I(z_1, \dots, z_m) = (- \bz_2, \bz_1, -\bz_4, \bz_3, \dots, -\bz_{2m}, \bz_{2m-1}).
\end{align*}
Then from~\eqref{eq:eigsp_qharm} we have $L_{-1} = I(L_0)$, where we recall that $L_{-1}$ is the image of $\bar\del\Psi$, see Section~\ref{sec:harm}. 
\begin{remark}
\label{quaternionic} The map $I$ can be viewed as the quaternionic structure on $\CC^{2m}$. Indeed, $\CC^{2m}$ is isomorphic to the quaternionic space $\mathbb{H}^m$ with the identification $(z_1, z_2) \to z_1+ {\bf j} z_2$.
Then the map $I$ is the multiplication by the element $\bf{j}$ on the right. 
We also note that the condition $L_{-1}=I(L_0)$  appears naturally in the context of harmonic maps to quaternionic projective spaces, see \cite[p. 284]{Ud97} and Remark \ref{rem:FLPP}.
\end{remark}

Furthermore, writing $\psi_j = (f_{j+} s_0, \bar{f}_{j-} \bar{s}_0)$ and $F = (f_{1+}, f_{1-}, \dots, f_{m-})$,  we have
\begin{align*}
    \left[\bar\del \Psi(\del_\bz)\right](F) = \frac{\lambda}{2}|F|^2 I(F),
\end{align*}
and hence all the zeroes of $\bar\del \Psi$ correspond to the zeroes of $F$ (i.e., the common zeroes of $\psi_1, \dots, \psi_m$), and are of even order. 

This leads to the following definition.
\begin{definition} 
\label{def:spinharm} 
A harmonic map $\Psi\colon (M,\mC) \to \CP^{2m-1}$ is called {\em quaternionic harmonic} if
    \begin{enumerate}
        \item $L_{-1} = I(L_0)$,
        \item all the zeroes of $\bar\del \Psi$ are of even order.
    \end{enumerate}
\end{definition}
The geometric meaning behind this definition and examples are discussed in Section~\ref{sec:qharm} below. We have seen that all harmonic maps by eigenspinors arising from the critical points of Dirac eigenvalues are quaternionic harmonic. The following proposition establishes the converse. 

\begin{proposition}
\label{prop:H_harm}
Suppose that $\Psi\colon M \to \CP^{2m-1}$ is a quaternionic harmonic map and let $D = (\bar\del\Psi)$ be the zero divisor of $\bar\del\Psi$. Then $\Psi$ induces a natural spin structure $S_\Psi = L_0^*\otimes\left[\frac{1}{2}D\right]$ on $M$. Furthermore, for a metric $g_\Psi = |\bar\del\Psi|_g^2g$ ($g\in \mC$ is arbitrary),  one can choose  a collection of 
$\mD_{g_\Psi}$--eigenspinors  $\psi_1,\ldots, \psi_m$ such that 
$$
 \Psi = [\psi_{1+} : \bar{\psi}_{1-} : \dots : \psi_{m+} : \bar{\psi}_{m-}]
 $$
 and
 \begin{equation}
 \label{eq:sumpsi}
 \sum_{j=1}^m {|\psi_j|_{g_\Psi}^2} = 1.
 \end{equation}
\end{proposition}

\begin{proof}

Since $\Psi$ is harmonic, $\bar\del \Psi\colon T^{(0,1)}M \to \Hom(L_0, L_{-1})$ is an anti-holomorphic morphism of line bundles.

We have the holomorphic isomorphism $\iota\colon L_0^* \to I(L_0)$, given for any $\alpha \in L_0^*$ by 
    \begin{align*}
        \alpha(l) = \langle \iota(\alpha), I(l)\rangle,
    \end{align*}
    and by assumption, $L_{-1} = I(L_0)$. Thus, $\Hom(L_0, L_{-1}) \cong L_0^* \otimes I(L_0) \cong L_0^* \otimes L_0^*$. Furthermore, $\bar\del \Psi\colon T^{(0,1)}M \to \Hom(L_0, L_{-1})$ is an anti-holomorphic map. Composing the two, one has an anti-holomorphic map
    $$
    T^{(0,1)}M \to L_0^*\otimes L_0^*
    $$
    whose dual is
    $$
    L_0\otimes L_0 \to \bar K.
    $$
Taking the complex conjugate of both sides, one has a holomorphic map 
    $$
    \overline{L_0\otimes L_0} \to K
    $$
and, finally, using the metric on $L_0$, we arrive at a holomorphic map
    $$
    A\colon L_0^*\otimes L_0^* \to K
    $$
    given by 
    \begin{equation}
    \label{def:A}
    A(s_1,s_2) = dz\quad\Longleftrightarrow\quad \langle (\bar\del_0 \Psi)(\del_{\bar z})(\overline{s_1^\flat}), I(\overline{s_2^\flat}) \rangle = 1,
    \end{equation}
    where $s_i^\flat\in \bar L_0$ is defined via $s_i(f) = \langle f, \overline{s_i^\flat}\rangle$, $f\in L_0$. In particular, if $f_i\in L_0$ is dual to $s_i$, i.e. $s_i(f_i) = 1$, then $f_i = \dfrac{\overline{s_i^\flat}}{|s_i^\flat|^2} = |f_i|^2f_i = \frac{1}{|s_i|^2}f_i$ and 
    \begin{equation}
    \label{def:B}
    \langle(\bar\del \Psi)(\del_{\bar z})(f_1),I(f_2) \rangle = |f_1|^2|f_2|^2.
    \end{equation}
    Similarly, if $s_0$ is a local meromorphic section satisfying $A(s_0,s_0) = dz$, then $\tilde s_0$, the dual of $\overline s_0^\flat$, is a local anti-meromorphic of $L_0^*$ section satisfying $(\bar\del \Psi)(\del_{\bar z}) = \tilde s_0\otimes \tilde s_0$. Since the poles of $s_0$ correspond to zeroes of $\tilde s_0$, we conclude that the zeroes of $A$ have the same order as the zeroes of $\bar\del\Psi$.
 
  Let $D$ be the zero divisor of $A$ as a section of $L_0\otimes L_0 \otimes K$, i.e.
    $$
    D := (A) = \sum_{p \in M} \mathrm{ord}_p(A) p
    $$
and let $[D]$ be the associated holomorphic line bundle. By construction 
$$
[D]\cong L_0\otimes L_0\otimes K,
$$
hence, combining it with $A$ yields a holomorphic isomorphism
$$
\tilde A\colon \left(L_0^*\otimes \left[\frac{1}{2}D\right]\right)\otimes \left(L_0^*\otimes \left[\frac{1}{2}D
\right]\right)\xrightarrow{\sim} K,
$$ 
where we used that the zeroes of $\bar\del_0\Psi$ (and hence of $A$) are of even order in order to define $\left[\frac{1}{2}D\right]$. In particular, this implies that $S_\Psi :=S=L_0^*\otimes \left[\frac{1}{2}D\right]$ is a spin structure on $M$. 

Let $z$ be a local coordinate on $M$. Define $s_0$ to be a local holomorphic section of $S$ such that $\tilde A(s_0\otimes s_0) = dz$,  and consider $F_0$ to be a (non-vanishing) holomorphic section of $S^* = L_0\otimes \left[-\frac{1}{2}D\right]$ such that $s_0(F_0) = 1$. If $\eta$ is a (global) holomorphic section of $\left[\frac{1}{2}D\right]$ such that $(\eta) = \frac{1}{2}D$, then one can write $s_0 = s \otimes \eta$, $F_0 = F\otimes \eta^{-1}$, where $s, F$ are local sections of $L_0^*$ and $L_0$, respectively,  satisfying $s(F) = 1$. In particular, $F$ is a local holomorphic section of $L_0$ satisfying the following three conditions:
\begin{enumerate}
\item $(F) - \frac{1}{2}D\geq 0$
\item as a section of $L_0$, $F$ is a local $\mathbb{C}^{2m}$-valued function such that $F$ is a local lift of $\Psi$ outside of the support of $D$. In particular, since $s(F) \equiv 1$, we see that $F\otimes s_0 = F\otimes s\otimes \eta$ is a globally defined section of $\mathbb{C}^{2m}\otimes S$, whose projectivisation is equal to $\Psi$ outside of the support of $D$. Thus, the map $\Psi$ is a map by spinors in the sense of~\eqref{eq:mapcp}.
\item Since $F$ is holomorphic, one has $\del_{\bar z} F\perp F$. Hence, $\left[\bar\del \Psi(\del_{\bar z})\right](F)=\del_{\bar z}F$, and by~\eqref{def:B} one has
$$
\del_{\bar z}F = |F|^2 I(F).
$$
If we define the metric on $M$ (with conical singularities at zeroes of $F$) by $2|s_0|_g^2 = |F|^{-2}$, then~\eqref{eq:eigsp_qharm} implies that $\Psi$ is a map by eigenspinors with eigenvalue $\lambda = 1$ and $2|F\otimes s_0|^2=1$. Moreover,~\eqref{eq:eigsp_qharm} implies that 
$$
|\del\Psi|_g^2 = 4\frac{|\del_{\bar z}F|^2}{|F^2||\del_z|_g^2} = 4|F|^4|s_0|_g^2 = 1,
$$
i.e. $g=g_\Psi$. Changing $F$ to $\tilde F = \sqrt{2}F$ we obtain that $|\tilde F|^2|s_0|^2 = 1$ and $\tilde F\otimes s_0$ defines the same map $\Psi$. Defining $\psi$ as in  \eqref{eq:psi} using $\tilde F$ instead of $F$ yields  \eqref{eq:sumpsi}.
Finally, by Lemma~\ref{lem:extend} any map by eigenspinors can be extended to the set of common zeroes, and this extension has to coincide with $\Psi$ on the support of $D$ by continuity. 
\end{enumerate}

\end{proof}

We summarise the results of this section in the following proposition.

\begin{proposition}\label{prop:spin-harmonic}
	Let $M$ be an oriented surface and $\mC$ be a conformal class of metrics on $M$. Suppose that $g\in \mC$ is $\bar\lambda_k$-conformally critical metric for a spin structure $S$. Then there exists a collection $\psi_j = (\psi_{j+},\psi_{j-})$, $j=1,\ldots, m$ of $\lambda_k$-eigenspinors,  such that the map $\Psi\colon(M,\mC)\to\CP^{2m-1}$ given in homogeneous coordinates by
$$	
\Psi = [\psi_{1+} : \bar\psi_{1-}: \dots : \psi_{m+}: \bar\psi_{m-}]
$$	
is a quaternionic harmonic map. Moreover, $\bar\lambda_k(M,g,S)^2 = E^{(0,1)}_g(\Psi)$ and $g = \alpha g_\Psi = \alpha |\bar\del\Psi|_g^2g$ for some $\alpha>0$.	

Conversely, let $\Psi\colon (M,\mC)\to\CP^{2m-1}$ be a quaternionic harmonic map, and let $S_\Psi$ be the spin structure induced by $\Psi$ in the sense of Proposition~\ref{prop:H_harm}. Then the metric $g_\Psi = |\bar\del\Psi|_g^2g$ is $\lambda_k$-conformally critical metric for a spin structure $S_\Psi$, where $k-1$ is the number of eigenvalues $\lambda_j(M,g_\Psi,S_\Psi)$ satisfying $0<\lambda_j(M,g_\Psi,S_\Psi)<1$. Furthermore, one has $\bar\lambda_k(M,g_\Psi,S_\Psi)^2 = E^{(0,1)}_\mC(\Psi)$.
\end{proposition}

\subsection{Examples of quaternionic harmonic maps}
\label{sec:qharm}
The condition (2)  in Definition~\ref{def:spinharm} is fairly  geometric, so in this section we mainly focus on the geometric meaning of the condition (1).
The simplest case is $m=1$, for which $I(L_0)=L_0^\perp$ and  the condition $L_{-1} = I(L_0)$ reduces to $\bar\del\Psi
\not\equiv 0$, i.e. $\Psi$ is a non-holomorphic harmonic map. In particular, any non-constant anti-holomorphic map satisfies condition (1) of Definition~\ref{def:spinharm}. 

If $m>1$, it is possible to construct maps satisfying $L_{-1} = I(L_0)$ from holomorphic maps to $\CP^{2m-1}$ using~\cite[Section 4]{BW}. To any linearly full holomorphic map $\Phi\colon M\to\CP^{2m-1}$ one can associate its Fr\'enet frame $\{\Phi=\Phi_0,\Phi_1,\ldots,\Phi_{2m-1}\}$, where $\Phi_j\colon M\to\CP^{2m-1}$ are mutually orthogonal and satisfy $\mathrm{Span}\{\Phi_0,\ldots,\Phi_j\} = \mathrm{Span}\{\Phi,\del_z\Phi,\ldots,\del_z^j\Phi\}$ as lines in $\CC^{2m-1}$. This is a  special example of a harmonic sequence~\cite{Wolfson}, in particular, $\Phi_j$ is harmonic for all $j=0,\ldots, 2m-1$ and $\Phi_{2m-1}$ is anti-holomorphic. In~\cite[Theorem 4.2]{BW} it is shown that if $\Phi_{2m-1} = I(\Phi_0)$, then $\Phi_{2m-1-j} = I(\Phi_j)$ for all $j=0,\ldots, 2m-1$ and, in particular, $\Phi_{m-1} = I(\Phi_m)$. Setting $\Psi = \Phi_m$ yields $L_{-1} = \mathrm{Span}\{\Phi_{m-1}\}$, thus, $\Psi$ is a harmonic map satisfying condition (1) of Definition~\ref{def:spinharm}. For $m=2$, the  relation $\Phi_{2m-1} = I(\Phi_0)$ is satisfied if and only if  $\Phi$ is {\em horizontal} with respect to the projection $\pi\colon\CP^{3}\to\HP^{1}$, i.e.  $\mathrm{Span}\{\Phi,\del_z\Phi\}\perp I(\Phi)$, see~\cite[Theorem 4.5]{BW}. Indeed, differentiating $\la \del_z\Phi, I(\Phi)\ra = 0$ yields $\la \del^2_z\Phi, I(\Phi)\ra = 0$, hence, $I(\Phi)= I(\Phi_0) = \Phi_3$.

\begin{example}
Consider a holomorphic map $\Phi\colon\CP^1\to \CP^3$ given by $\Phi([1:z]) = [1:z^3:-\sqrt{3}z:\sqrt{3}z^2]$. It is easy to check that $\Phi$ is horizontal, $\Phi_1 = [-3\bar z:3z^2:\sqrt{3}(2|z|^2-1):\sqrt{3}z(2-|z|^2)]$, $\Phi_2 = I(\Phi_1)$. Setting $\Psi = \Phi_2$ gives a harmonic map satisfying condition (1) of Definition~\ref{def:spinharm}. One computes
$$
g_\Psi = |\bar\del\Psi|_g^2g = 4\frac{|\la\del_z\Phi_1,I(\Phi_1)\ra|^2}{|\Phi_1|^4}dzd\bar z = \frac{16 dzd\bar z}{(1+|z|^2)^2} = 4g_{\mathbb{S}^2}.
$$
Hence, $\bar\del\Psi$ has no zeroes and $\Psi$ is a quaternionic harmonic map. Furthermore, this computation shows that the round metric on the sphere is critical for a Dirac eigenvalue $\lambda_k$ such that $\bar\lambda_k(\mathbb{S}^2,g,S)^2 = 16\pi$, so that $\lambda_k=2$, i.e. $k=2$. Therefore, we have established that the round metric on the sphere is $\bar\lambda_2$-critical.
\end{example}

\begin{example}
The previous example is a special case of the so-called {\em Veronese sequence}~\cite[Section 5]{BJRW}. Namely, for any $m$ one can consider a holomorphic map $\Phi\colon\mathbb{CP}^1\to \CP^{2m-1}$ given by
$$
\Phi([1:z]) = \left[1:\sqrt{2m-1}z:\ldots:\sqrt{\binom{2m-1}{j}}z^j:\ldots:z^{2m-1}\right].
$$ 
In~\cite[Theorem 5.2]{BJRW} it is computed that 
$$
\Phi_{2m-1}([1:z]) = \left[-\bar z^{2m-1}:\ldots:(-1)^{2m-1-j}\sqrt{\binom{2m-1}{j}}\bar z^{2m-1-j}:\ldots:1\right].
$$
Therefore, $\Phi_{2m-1} = \tilde I(\Phi_0)$, where
$$
\tilde I([z_0:\ldots:z_{2m-1}]) = [-\bar z_{2m-1}:\bar z_{2m-2}:\ldots:(-1)^{2m-1-j}\bar z_{2m-1-j}:\ldots:1].
$$
Consider $A\in \mathrm{PU}(2m)$ given by
\begin{align*}
&A([z_0:\ldots:z_{2m-1}]) = \\
&[z_0:z_{2m-1}:-z_1:z_{2m-2}:\ldots:(-1)^jz_j:z_{2m-1-j}:\ldots: (-1)^{m-1}z_{m-1}:z_m],
\end{align*}
then $A\tilde I = IA$, so that  $A\Phi_{2m-1} = I(A\Phi_0)$. Since $A$ is an isometry, one has $(A\Phi)_j = A\Phi_j$, hence $\Psi = A\Phi_m$ satisfies $L_{-1} = I(L_0)$. Furthermore, one computes using notation from~\cite[Theorem 5.2]{BJRW} that
$$
g_\Psi = 4\gamma_{m-1}dzd\bar z = m^2 g_{\mathbb{S}^2}.
$$
As a result, $\bar\del\Psi$ has no zeroes and $\Psi$ is a quaternionic harmonic map. In fact, one sees that the map $\Psi$ is a map by Dirac eigenspinors on the round sphere associated with the  eigenvalue $\lambda_k = m$, i.e. $k = \frac{m(m+1)}{2}+1$.
\end{example}
\begin{remark}
\label{rem:FLPP}
 The notion of quaternionic harmonic maps is related to the notion of quaternionic holomorphic maps studied in~\cite{FLPP}. Consider the twistor projection $\pi\colon\CP^{2m-1}\to\HP^{m-1}$, then it follows from~\cite[Lemma 2.7]{FLPP} that $\Psi\colon M\to\CP^{2m-1}$ satisfies $L_{-1} = I(L_0)$ if and only if $\pi\Psi$ is a quaternionic holomorphic map. The latter is equivalent to saying that $\pi_*\bar\del\Psi = 0$, i.e. the $\bar\del$-derivative of $\Psi$ is vertical .
\end{remark}

\section{$\Lambda_1(M,\mC,S)$ as minimum of energy} 
\label{sec:energyproof}

In the present section we aim to obtain a characterisation for the quantity $\Lambda_1(M,\mC,S)$ in terms of the variational theory for the energy functional. We are motivated by the Hersch's inequality~\cite{Hersch} and the min-max characterisation for maximisers of Laplace eigenvalues~\cite{KS}. Since we are dealing exclusively with the first eigenvalue, Remark~\ref{rmk:lambda1_min} implies that it is sufficient to consider maps to $\CP^1$.

\subsection{Energy of maps by eigenspinors}
\label{subsec:energy}
For a given spin structure $S$ on $M$ and a conformal class $\mathcal{C}$, we consider the set 
\begin{align*}
    \Gamma_{\mathcal{C},S} = \{\Psi = [\psi_+ : \bar\psi_-] : M \to \CP^1 \,|\, \exists g \in \mathcal{C}, \lambda \neq 0, \D_g\psi = \lambda \psi\}
\end{align*}
of all maps obtained from eigenspinors (with non-zero eigenvalues) as described in the previous section.

\begin{lemma} \label{lem:U-energy}
    Let $\psi = (\psi_+, \psi_-)$ be an eigenspinor on $(M, S, g)$, with eigenvalue $\lambda \neq 0$. Let $\Psi = [\psi_+ : \bar\psi_-]\colon M \to \CP^1$, so that $\Psi\in\Gamma_{\mC,S}$. Then
    \begin{align*}
        E^{(0,1)}(\Psi, [g]) = \lambda^2 \vol(M,g).
    \end{align*}
 In particular, for any conformal class $\mC$ and any spin structure $S$ on $M$ one has
 \begin{equation}
 \label{eq:energy_gamma}
  \Lambda_1(M,\mC,S)^2 = \inf_{\Psi \in \Gamma_{\mC,S}} E^{(0,1)}(\Psi, \mathcal{C}).
 \end{equation}
\end{lemma}
\begin{proof}
    Let $Z = \{p \in M\colon \psi(p) = 0\}$ (note that it is a discrete set). Let 
    \begin{align*}
        U(\psi,g) = \int_{M \setminus Z} \frac{|D_g \psi|_g^2}{|\psi|_g^2} dv_g,
    \end{align*}
    so that $U(\psi,h) = U(\psi,g)$ for any $h \in [g]$. 

    Define the metric $g_\psi = |\psi|^4_g g$ on $M\setminus Z$,  so that $|\psi|^2_{g_\psi} = 1$. On one hand, we have
    \begin{align*}
        U(\psi, g_{\psi}) &= \int_{M \setminus Z} |D_{g_{\psi}} \psi|^2_{g_{\psi}} dv_{g_\psi} \\
                          &= \int_{M \setminus Z} |D_g \psi|^2_{g_{\psi}} dv_g \\
                          &= \int_{M \setminus Z} \lambda^2 |\psi|^2_{g_\psi} dv_g \\
                          &= \lambda^2 \vol(M\setminus Z, g) = \lambda^2 \vol(M,g).
    \end{align*}

    On the other hand, let $s_0$ be a local holomorphic section of $S$ such that $s_0 \otimes s_0 = dz$ and write $\psi = (f_+ s_0, \bar{f}_- \bar{s}_0), F= (f_+, f_-)$. Then
    \begin{align*}
        \frac{|D_g \psi|_g^2}{|\psi|_g^2} &= 4 \frac{|s_0|_g^6 |\del_\bz F|^2}{|s_0|^2 |F|^2}
        = 4 \frac{|\del_\bz F|^2}{|\del_\bz|^2_g |F|^2}
    \end{align*}
    and since $\psi$ is an eigenspinor, $\del_\bz F \perp F$ so
    \begin{align*}
        \frac{|D_g \psi|_g^2}{|\psi|_g^2} = 4 \frac{|\pi_{F^\perp} \del_\bz F|^2}{|\del_\bz|^2_g |F|^2} = |\bar\del \Psi|^2_g,
    \end{align*}
    where we used~\eqref{eq:FS} in the last equality. Integrating we obtain
    \begin{align*}
        U(\psi, g) = \int_{M\setminus Z} |\bar\del \Psi|^2 dv_g = E^{(0,1)}(\Psi, [g]).
    \end{align*}
    Taking the infimum over all maps in $\Gamma_{\mC,S}$ yields~\eqref{eq:energy_gamma}.
\end{proof}

As a first application, we give a more geometric proof of~\eqref{ineq:Bar}.

\begin{theorem}[\cite{B0}] \label{thm:bar}
 Let $M=\mathbb{S}^2$ with its unique spin structure $S$. Then for any Riemannian metric $g$ one has
    \begin{align*}
        \lambda_1^2(M, g,S) \area(M,g) \geq 4 \pi,
    \end{align*}
    with equality if and only if $g$ is a round metric.
\end{theorem}

\begin{proof}
    Write $S$ for the unique spin structure on $M=\mathbb{S}^2$. Let $\psi$ be a first eigenspinor, $\D_g \psi = \lambda_1 \psi$. It is known that $\lambda_1 > 0$. Let $\Psi = [\psi_- : \bar\psi_+] : M \to \CP^1$ be its associated map. Then
    \begin{align*}
        \mathrm{Index}(\bar\del \Psi) = -\chi(M) - \deg(\Psi) \chi(\CP^1)
    \end{align*}
    where $\mathrm{Index}(\bar\del \Psi)$ is the total number of zeroes of $\bar\del \Psi$ counted with multiplicity, $\chi$ is the Euler characteristic and $\deg$ is the mapping degree (see e.g. \cite{EW} for the proof).
    Thus,
    \begin{align*}
        \deg (\Psi) = - \frac{\chi(M) + \mathrm{Index}(\bar\del \Psi)}{\chi(\CP^1)} \leq -1.
    \end{align*}
    Set $d = -\deg(\Psi) \geq 1$.
    By Lemma \ref{lem:U-energy}, 
    \begin{align*}
        E^{(0,1)}(\Psi) = \lambda_1^2 \area(\mathbb{S}^2, g).
    \end{align*}
Combining with~\eqref{eq:Edeg} one has
$$
\lambda_1^2 \area(\mathbb{S}^2, g) = 4\pi d + E^{(1,0)}(\Psi)\geq 4\pi.
$$
    Suppose we have the equality $\lambda_1^2 \area(\mathbb{S}^2, g) = 4\pi$. Then $d = 1$ and $E^{(1,0)}(\Psi) = 0$, i.e.
$\Psi$ is an anti-holomorphic diffeomorphism and, in particular, $\Psi$ is conformal. Hence, $\frac{1}{2}|d\Psi|_g^2g = \Psi^* g_{\CP^1}$ and 
    \begin{align*}
        \frac{1}{2}|d\Psi|_g^2 = |\del \Psi|_g^2 + |\bar\del \Psi|_g^2 = |\bar\del \Psi|_g^2 = \lambda_1^2(g).
    \end{align*}
    So $g$ is the round metric up to a constant scaling.
\end{proof}

Another possible way to use formula~\eqref{eq:energy_gamma} is to apply calculus of variations techniques in order to show that the infimum on the right hand side is achieved on some critical point of the functional $E^{(0,1)}$, i.e. a harmonic map. This would provide an alternative way to establish the existence of minimising metrics, i.e. to reprove Theorem~\ref{theo:Am}. However, the definition of the set $\Gamma_{\mC,S}$ is not particularly geometric, which makes such an approach difficult. Instead, below we use Theorem~\ref{theo:Am} to provide a more geometric formulation of~\eqref{eq:energy_gamma}. 

We first observe that any non-holomorphic harmonic map to $\CP^1$ satisfies condition (1) of Definition~\ref{def:spinharm}. Using this we can define 
\begin{align*}
    H_{\mathcal{C},S} &= \left\{\Psi : (M,C)\to \CP^1 \,|\, \text{$\Psi$ is harmonic, $\Psi$ is non-holomorphic,} \right.\\
                      &\left.\quad \text{all the zeroes of $\bar\del_0 \Psi$ are of even order, and  $\Psi^*L^* \otimes \left[\frac{1}{2}(\bar\del_0 \Psi)\right] \cong S$}\right\},
\end{align*}
where the isomorphism on the right-hand side is given by Proposition~\ref{prop:H_harm}. In particular, Proposition~\ref{prop:H_harm} implies that $H_{\mC,S}\subset \Gamma_{\mC,S}$. The following theorem states that the infimum on the right-hand side of~\eqref{eq:energy_gamma} is achieved on an element of $H_{\mathcal C,S}$, up to a formation of a bubble. 

\begin{theorem}\label{thm:energy}
    Let $S$ be a spin structure on $M$ and $\mathcal{C}$ be a conformal class. Then
    \begin{align*}
        \Lambda_1(M,\mathcal{C},S)^2 = \min \left\{\inf_{\Psi \in H_{\mathcal{C},S}} E_\mC^{(0,1)}(\Psi), 4\pi\right\}.
    \end{align*}
    In particular, if $\inf_{\Psi \in H_{\mathcal{C},S}} E_\mC^{(0,1)}(\Psi)<4\pi$, then there exists $\Psi\in H_{\mC,S}$ such that $\Lambda_1(M,\mC,S)^2 = E_\mC^{(0,1)}(\Psi)$. 
\end{theorem}

\begin{proof}

    It is known from~\cite{GH} that $\Lambda_1(M,\mathcal{C},S)^2 \leq 4\pi$, so it suffices to prove the equality in the case $\inf_{\Psi \in H_{\mC,S}} E^{(0,1)} (\Psi) < 4\pi$. In this case, let $\Psi \in H_{\mathcal{C}, S}$ such that $E^{(0,1)}(\Psi) < 4\pi$. Since $H_{\mathcal{C},S}$ corresponds to the quaternionic harmonic maps, by Proposition \ref{prop:spin-harmonic} there exists a metric $g_\Psi$ and eigenspinor $\psi$, $\D_{g_\Psi} \psi = \lambda \psi$, such that $\Psi = [\psi_{+} : \bar\psi_{-}]$. Then by Lemma \ref{lem:U-energy}, 
    \begin{align*}
        4\pi > E^{(0,1)}(\Psi) = \lambda^2 \vol(M,g_\Psi) \geq \Lambda_1(M,\mathcal{C},S)^2.
    \end{align*}
    It  follows from Theorem~\ref{theo:Am}  that there exists a conformally minimal metric $g$ with $\Lambda_1(M,\mathcal{C},S) = \bar\lambda_1(M,g,S)$. By Theorem \ref{prop:harmo}, we obtain a quaternionic harmonic map $\Phi\colon M \to \CP^1$ induced by an eigenspinor with eigenvalue $\lambda_1(M,g,S)$. So $\Phi \in H_{\mathcal{C},S}$ and 
    \begin{align*}
        \Lambda_1(M,\mathcal{C},S)^2 = \bar\lambda_1(M,g,S)^2 = E^{(0,1)}_\mC(\Phi) \geq \inf_{\Psi \in H_{\mC,S}} E_\mC^{(0,1)}(\Psi).
    \end{align*}
    The reverse inequality is a consequence of Lemma \ref{lem:U-energy} combined with the inclusion $H_{\mC,S}\subset \Gamma_{\mC,S}$.
\end{proof}

\subsection{Minimisers on the torus}
\label{subsec:torus}

In the last part of this section, we use the previous characterisation of conformally extremal metrics in terms of harmonic maps to prove Theorem \ref{thm:torus}. As an example,  and because this result will be needed later, we compute the Dirac eigenvalues on the flat torus. We follow the method used in \cite[Section 2.1]{Ginoux}.

For any vector $(a,b)\in \mathbb{R}^2$, $b\ne 0$ one has a lattice $\Gamma=\Gamma_{(a,b)} = \mathbb{Z}(1,0) + \mathbb{Z}(a,b)$ and the corresponding flat torus $\mathbb{T}^2 = \mathbb{R}^2/\Gamma_{a,b}$ with the flat metric $g_{(a,b)}$ induced by the projection map. The spin structure is given by a homomorphism $\chi\colon \Gamma \to \ZZ/2\ZZ$, so that the global sections of the spinor bundle $S$ can be identified with complex functions $\psi$ on $\RR^2=\C$ satisfying 
\begin{equation}
\label{eq:chi_period}
\psi(z+\gamma) = (-1)^{\chi(\gamma)}\psi(z)
\end{equation}
for all $\gamma\in\Gamma$, $z\in\C$. The canonical bundle on $\mathbb{T}^2$ is trivialised by a global section $dz$ and the isomorphism $S\otimes S\to K$ is given by $\psi_1\otimes\psi_2\mapsto \psi_1\psi_2\,dz$. Since $|dz|_{g_{a,b}} = 1$, $\bar S$ is identified with $S$ via $\bar S\otimes K\to S$. With these identifications, a global section of $S\oplus \bar S$ consists of two complex functions $\psi_+,\psi_-$ on $\C$, satisfying~\eqref{eq:chi_period} and the Dirac operator with respect to the flat metric $g_{a,b}$ is given by 
\begin{align*}
    \D\begin{pmatrix} \psi_+ \\ \psi_- \end{pmatrix} &= 2 \begin{pmatrix}
        0 & \del_z \\
        - \del_{\bar z} & 0 
    \end{pmatrix}
        \begin{pmatrix} \psi_+ \\ \psi_- \end{pmatrix}.
\end{align*}

The spectrum of $\mD$ is best described in terms of the dual lattice $\Gamma^*$. Recall that $\Gamma^*$ is the lattice in the dual space $(\RR^2)^*$ defined as $\gamma^*\in \Gamma^*$ if and only if $\gamma^*(\gamma)\in \ZZ$ for all $\gamma\in \Gamma$. For a homomorphism $\chi\colon \Gamma \to \ZZ/2\ZZ$, an affine lattice $\Gamma^*_\chi$ is defined as follows: $\xi\in \Gamma_\chi^*$ if and only if  $\xi(\gamma) + \frac{1}{2}\chi(\gamma)\in \ZZ$. It is easy to see that $\Gamma^*_\chi = \Gamma^* + \eta$, where $\eta = \frac{1}{2}\sum\limits_{j=1,2}\chi(\gamma_j)\gamma_j^*$ for some basis $(\gamma_1,\gamma_2)$ of $\Gamma$ and the dual basis $(\gamma_1^*,\gamma_2^*)$ of $\Gamma^*$.

\begin{lemma} \label{lem:dirac-vp}

    Let $S$ be a spin structure on $(\mathbb{T}^2,g_{(a,b)})$ and $\chi\colon \Gamma \to \ZZ/2\ZZ$ be its associated homomorphism. 

    The Dirac spectrum on $(\mathbb{T}^2,g_{(a,b)})$ is given by
    \begin{align*}
        \mathrm{Spec}(\D) = \left\{\pm 2 \pi |\xi| : \xi \in \Gamma^*_\chi \right\}.
    \end{align*}
\end{lemma}

\begin{proof}

    For a given constant spinor $\phi = (\phi_+, \phi_-)$ on $\mathbb{R}^2$ and $\xi \in \Gamma_\chi^*$ consider the following spinor on $\mathbb{R}^2$
    \begin{align*}
        \psi(x,y) = e^{2\pi i \xi(x,y)} \phi,
    \end{align*}
    By construction, this spinor satisfies the periodicity conditions and corresponds to a spinor on $\mathbb{T}^2$. 

    Complexifying $\xi\in (\RR^2)^*$ gives an element in $(\C^2)^*$ which we denote by the same letter $\xi$. A quick calculation gives
    \begin{align*}
        \D \psi(x,y) = \frac{4 \pi i}{2} \underbrace{\begin{pmatrix}
                0 & \xi(1,-i) \\
                - \xi(1,i) & 0 
        \end{pmatrix}}_{\mathbf{M}} \psi(x,y).
    \end{align*}
    Thus, the spinor $\psi$ is an eigenspinor if it is an eigenvector of $\mathbf{M}$, which happens if and only if $\phi$ has the same property. It is easy to see that the eigenvalues of $\mathbf{M}$ are $\pm i |\xi|$, hence
    \begin{align*}
        \{\pm 2 \pi |\xi| : \xi \in \Gamma_\chi^*\} \subset \mathrm{Spec}(\D).
    \end{align*}
    To have the equality between the two sets, it suffices to remark that $\{e^{2\pi i\gamma^*} : \gamma^* \in \Gamma^*\}$ is a basis of $L^2(\mathbb{R}^2/\Gamma)$, so the spinors $\psi_j$, constructed from a basis of eigenvectors $\phi_j$ of $\mathbf{M}$, form a basis of $L^2(\Sigma T)$ when considering all $\gamma^* \in \Gamma^*$. 
\end{proof}

Finally, let us describe the moduli space of flat tori with spin-structures. The spin structure $S_{0}$ associated with a trivial homomorphism $\chi$ is accordingly called trivial. The moduli space of flat tori endowed with $S_{0}$ coincides with the usual  moduli space of oriented flat tori, i.e. one can always arrange that $|a|\leq 1/2$ and $a^2+b^2\geq 1$. The Dirac operator for this structure has a non-trivial kernel corresponding to constant spinors on $\RR^2$. For those values of $a,b$ one has 
\begin{equation}
\label{eq:lambda1_S0}
\lambda_1(\mathbb{T}^2,g_{a,b}, S_0) = \frac{2\pi}{b},\qquad\bar\lambda_1(\mathbb{T}^2,g_{a,b}, S_0) = \frac{2\pi}{\sqrt{b}}.
\end{equation}

All the other spin structures are isomorphic after a change of basis of the lattice $\Gamma$, this spin-structure is referred to as {\em non-trivial} and is denoted simply by $S_1$. In this case we can arrange that $\chi(1,0) = 0$, $\chi(a,b) = 1$, $|a|\leq 1/2$ and $b^2+(|a|-1/2)^2\geq 1/4$. The Dirac operator for $S_1$ is invertible. 
For those values of $a,b$ one has 
$\lambda_1(\mathbb{T}^2,g_{a,b}, S_1) = \frac{\pi}{b}$ and $\bar\lambda_1(\mathbb{T}^2,g_{a,b}, S_1) = \frac{\pi}{\sqrt{b}}$.

The main result of this section concides with Theorem \ref{thm:torustrivial} for $S_0$, and extends it 
to the case of $S_1$.
\begin{theorem} \label{thm:torus}
    Let $b_S = 2\pi$ if $S=S_0$ is the trivial spin structure on $\mathbb{T}^2$, and $b_S = \frac{\pi}{2}$ otherwise.  Then for all $b > b_S$ and all $a$, $|a|\le 1/2$,  one has
    $$
    \Lambda_1(\mathbb{T}^2,[g_{a,b}],S) = \bar\lambda_1(\mathbb{T}^2,g_{a,b}, S) = d_S\frac{\pi}{\sqrt{b}},
    $$
    where $d_S = 2$ if $S=S_0$ and $d_S=1$ otherwise. Furthermore, the flat metric is the unique smooth minimiser.
\end{theorem}

We recall the following theorem, which appears as~\cite[Corollary 6.6]{FLPP}. Together with Theorem~\ref{thm:energy}, it is the main ingredient in the proof of Theorem~\ref{thm:torus}.

\begin{theorem}[\cite{FLPP}] \label{thm:quater}
    If $F\colon \mathbb{T}^2 \to \mathbb{S}^2$ is a harmonic map with energy $E(F) < 4\pi$, then $F$ is a map to a circle $\mathbb{S}^1 \subset \mathbb{S}^2$.
\end{theorem}

\begin{proof}[Proof of Theorem \ref{thm:torus}]
    The proof proceeds in three steps: first Theorem~\ref{theo:Am} is used to show the existence of a conformally minimal metric. Then, from this metric and the corresponding $\lambda_1$-eigenspinors, we obtain  a harmonic map $\Psi\colon \mathbb{T}^2 \to \mathbb{S}^2$  and compute its energy to apply Theorem \ref{thm:quater}. Finally, we use that $\Psi$ is a harmonic map from the torus to $\mathbb{S}^1$ constructed from eigenspinors to show that the conformal factor $|dz| = |s_0|^2$ for the conformally minimal metric must be constant.

Let us first assume that the spin structure is trivial, i.e. $S=S_0$. Then by~\eqref{eq:lambda1_S0} we have that
$$
\Lambda_1(\mathbb{T}^2, [g_{(a,b)}], S_0)\leq \frac{2\pi}{\sqrt{b}}.
$$
Hence, for $b > \pi$, the existence condition in Theorem~\ref{theo:Am} is satisfied and there exists a minimiser $g_{\min}$ achieving $\Lambda_1(\mathbb{T}^2, [g_{(a,b)}], S_0)$. Furthermore, since the genus of $\mathbb{T}^2$ is $1$, this metric has no singularities.

For any such minimiser $g_{\min}$ by Proposition \ref{prop:harmo} and Remark~\ref{rmk:lambda1_min}, there exists an eigenspinor $\psi = (\psi_+, \psi_-)$ on $\mathbb{T}^2$ such that
    \begin{align*}
        \D_{g_{\min}} \psi = \lambda_1(\mathbb{T}^2,g_{\min},S_0) \psi \qquad \text{and} \quad |\psi|^2_{g_{\min}} = 1 \text{ on $\mathbb{T}^2$}.
    \end{align*}
 This eigenspinor gives rise to a harmonic map $\Psi\colon \mathbb{T}^2 \to \CP^1 = \mathbb{S}^2$,  and by Lemma \ref{lem:U-energy}
    \begin{align*}
        \lambda_1(\mathbb{T}^2,g_{\min},S_0)^2 \area(M,g_{\min}) = E^{(0,1)}(\Psi).
    \end{align*}
Note that by \cite{EW}, the degree $d$ of a harmonic map from the torus to the sphere is either $0$ or $|d| \geq 2$. In the second case, the map is holomorphic if $d>0$ or anti-holomorphic if $d<0$. Since $\Psi$ comes from an eigenspinor with $\lambda_1 \neq 0$, it cannot be holomorphic (as that would imply $\D \psi = 0$). At the same time,  if $\Psi$ is anti-holomorphic,  then $\lambda_1(\mathbb{T}^2,g_{\min},S_0)^2 = E^{(0,1)}(\Psi) = E(\Psi) = 4 |d| \pi > 4\pi$, which contradicts the upper bound on $\Lambda_1(\mathbb{T}^2, [g_{(a,b)}], S_0)$. Hence, $d=0$ and~\eqref{eq:Edeg} implies that
\begin{equation}
\label{eq:E_torus}
E(\Psi)  = 2 E^{(0,1)}(\Psi) + d  = 2E^{(0,1)}(\Psi) = 2\bar \lambda_1(\mathbb{T}^2,g_{\min},S_0)^2< 4\pi
\end{equation}
as long as $b>2\pi$.
 
Therefore, by Theorem~\ref{thm:quater}, $\Psi$ maps to a great circle $\mathbb{S}^1 \subset \mathbb{S}^2 = \CP^1$. Up to a rotation of the sphere we can assume that $\Psi = [\phi : 1]$, where $\phi\colon \mathbb{T}^2\to \mathbb{S}^1\subset\C$ is a harmonic map to a unit circle. By~\cite[Section 7]{EL} one has $\phi(x,y) = e^{2\pi i\gamma^*(x,y)}$, $\gamma^*\in\Gamma^*$. By Proposition~\ref{prop:harmo} one has that $g_{\min}$ is proportional to $g_\Psi = \frac{1}{2}|\bar\del\Psi|_{g_{(a,b)}}^2g_{(a,b)}$. Therefore, by~\eqref{eq:FS} with $F= (\phi,1)$ and $L=\C F$ we compute  
$$
g_\Psi = 2\frac{|\pi_{L^\perp}\partial_zF|^2}{|F|^2}g_{(a,b)} = 2\pi^2|\gamma^*|^2g_{(a,b)}.
$$
As a result, $g_{\min}$ is proportional to $g_{(a,b)}$, and hence, flat.

For the non-trivial spin structure $S_1$ the proof is the same with the main difference being
$$
\Lambda_1(\mathbb{T}^2, [g_{(a,b)}], S_1)\leq \frac{\pi}{\sqrt{b}},
$$
so that the existence condition is satisfied for $b>\pi/4$. The corresponding map $\Psi$ to $\CP^1$ has degree $0$ and energy 
$$
E(\Psi) = 2\bar \lambda^2_1(\mathbb{T}^2, [g_{(a,b)}], S_1) = \frac{2\pi^2}{b},
$$
which is smaller than $4\pi$ for $b>\pi/2$. Hence,  $\Psi$ is a map to $\mathbb{S}^2$ and the rest of the proof can be repeated verbatim.
\end{proof}
Concerning the other conformal classes, \eqref{eq:E_torus} implies that if a minimising metric exists, then it corresponds to a degree $0$ harmonic map $\Psi\colon \mathbb{T}^2\to \CP^1$ of energy less than $8\pi$. According to Proposition~\ref{prop:H_harm} any such map induces a spin structure on $\mathbb{T}^2$, either the trivial $S_0$ or the non-trivial $S_1$. For the trivial structure $S_0$, to the best of our knowledge, the only known maps with these properties are the maps to equator $\mathbb{S}^1\subset\CP^1$, which by~\eqref{eq:lambda1_S0} and~\eqref{eq:E_torus} have energy below $8\pi$ only for $b>\pi$. This motivates Conjecture \ref{conj:torus}.
\begin{remark}
For the non-trivial spin-structure $S_1$ and conformal classes $[g_{0,b}]$, $0<b<1$ there exist maps of degree $0$ into $\CP^1$ with energy smaller than $8\pi$. Those maps are Gauss maps of 
rotationally symmetric cmc-surfaces in $\RR^3$ called Delaunay unduloids, see \cite{Am}.  It was conjectured in \cite{Am}
that the corresponding (non-flat) metrics are coformally minimal for the first eigenvalue.   While the flat metric could be  a minimiser for some  $b \le \frac{\pi}{2}$, it is difficult  to formulate a more precise conjecture in the case of $S_1$. 
Note that it was shown in \cite{AmHu1} that as $b \to 0$, $\Lambda_1(\mathbb{T}^2, [g_{a,b}], S_1) \to 2\sqrt{\pi}$
which is consistent with the properties of the unduloids. It seems plausible that in contrast to the trivial spin structure (cf. Conjecture \ref{conj:torus}), in the case of $S_1$, smooth minimisers exist for all conformal classes. 
\end{remark}

\section{Minimal surfaces and critical metrics}
\label{sec:glob_crit}

\subsection{Standard definition of the Dirac operator} In the previous sections we used complex geometry to describe the Dirac operator. It is an especially convenient point of view in the situation where the conformal class is fixed, because varying metric only affects the metric on the spinor bundle. Once the conformal structure varies,  the spinor bundle varies as well, so the complex geometric picture is not always the most convenient and sometimes the traditional definition of Dirac operator is more appropriate. For that reason, in this section we give the standard definition of Dirac operator and show that it is equivalent to the one we used before.

Throughout this section $(M,g)$ is an oriented surface, $S$ is a spin structure, $\Sigma M = S\oplus\bar S$ is a spinor bundle and we consider the associated Hermitian metrics on $S$ and $\bar S$. The spin structure and the metric define the maps $\mu^{(1,0)}_g\colon K\otimes \bar S \to S$ and $\mu^{(0,1)}_g\colon\bar K\otimes S\to \bar S$. Combining $\mu^{(1,0)}_g$ with $\left(-\mu^{(0,1)}_g\right)$ and zero maps $K\otimes \bar S \to \bar S$, $\bar K\otimes S \to S$, yields the map $\mu_g\colon T_\C M^*\otimes (S\oplus \bar S)\to (S\oplus \bar S)$, which is sometimes called {\em Clifford multiplication}. We will reserve that term for the adjoint map.

\begin{definition}
    The {\it Clifford multiplication} is the $\CC$-bilinear map $(\cdot) : T_\CC M \otimes \Sigma M \to \Sigma M$ given by $X\cdot\psi = \mu_g(X^\flat\otimes \psi)$, where $X\mapsto X^\flat$ is an isomorphism between $T_\CC M$ and $T_\C M^*$ induced by the Hermitian metric, see~\eqref{def:A}. 
\end{definition}

In local coordinates, if $s_0$ is such that $s_0\otimes s_0 = dz$, then
    \begin{align*}
        \del_\bz \cdot \begin{pmatrix} f_+s_0 \\ f_-\bar s_0 \end{pmatrix}= \frac{1}{|s_0|^2} \begin{pmatrix} f_-s_0 \\ 0\end{pmatrix} \qquad
        \del_z \cdot \begin{pmatrix} f_+s_0 \\ f_-\bar s_0 \end{pmatrix} = \frac{1}{|s_0|^{2}} \begin{pmatrix} 0 \\ -f_+\bar s_0\end{pmatrix}.
    \end{align*}

\begin{lemma}
Clifford multiplication satisfies the following properties
\begin{enumerate}
\item $X\cdot(Y\cdot \psi) + Y\cdot(X\cdot\psi) = -2g_\C(X,Y)\psi$, where $g_\C$ is $\C$-bilinear extension of $g$ to $T_\C M$
\item For the Hermitian metric $\la\cdot,\cdot\ra$ on $S\oplus \bar S$  one has
$$
\la X\cdot\psi,\phi\ra = -\la \psi, \bar X\cdot \phi\ra.
$$
In particular, Clifford multiplication by a real tangent vector is a skew-Hermitian operator.
\end{enumerate}
\end{lemma}
\begin{proof}
(1) Writing $X = a\del_z + b\del_\bz$ and $Y= c\del_z + d\del_\bz$,  one has
  \begin{align*}
X\cdot(Y\cdot ) + Y\cdot(X\cdot) &= |s_0|_g^{-4}\begin{pmatrix} 0 & b\\ -a & 0 \end{pmatrix}\begin{pmatrix} 0 & d\\ -c & 0 \end{pmatrix} +|s_0|_g^{-4}\begin{pmatrix} 0 & d\\ -c & 0 \end{pmatrix} \begin{pmatrix} 0 & b\\ -a & 0 \end{pmatrix} = \\
& = -|s_0|_g^{-4}\begin{pmatrix} ad+bc & 0\\ 0 & ad+bc \end{pmatrix}.
  \end{align*}
  Observing that $g = |s_0|^{-4}(dx^2+dy^2)$ one has $g_\C(\del_z,\del_z)= g_\C(\del_\bz,\del_\bz) = 0$ and $g_\C(\del_z,\del_\bz)= g_\C(\del_\bz,\del_z) = \frac{|s_0|^{-4}}{2}$ completes the proof.
  
(2) By linearity and symmetry it is sufficient to check the equality for $X=\del_\bz$. Writing $\psi = (f_+s_0,f_-\bar s_0)$ and $\phi = (k_+s_0,k_-\bar s_0)$ one has
$$
|s_0|^2\la X\cdot\psi,\phi\ra = f_-\bar k_+ = - |s_0|^2\la \psi,X\cdot \phi\ra.
$$ 
\end{proof}

Recall that a Chern connection on a holomorphic bundle equipped with a Hermitian metric is a unique connection that preserves the metric and whose $(0,1)$-component coincides with $\bar\del$-operator~\cite{GrHa}.

\begin{definition}
The {\em Clifford connection} on $\Sigma M$ is defined as 
$$
\nabla^{Cl}_X(\psi_+,\psi_-):=(\nabla^{Ch}_X\psi_+, \overline{\nabla^{Ch}_{\bar X}\bar\psi_-}),
$$
where $\nabla^{Ch}$ is the Chern connection on $S$.
\end{definition}

Since Chern connection preserves the metric, Clifford connection does as well. In local coordinates,
\begin{align*}
        \nabla^{Cl}_{\del_\bz} \begin{pmatrix} f_+s_0 \\ f_-\bar s_0 \end{pmatrix} &=  \begin{pmatrix} (\del_{\bz} f_+)s_0  \\ (\del_{\bz}f_- + f_-\del_{\bz}\log|s_0|^2)\bar s_0\end{pmatrix} \\
        \nabla^{Cl}_{\del_z} \begin{pmatrix} f_+s_0 \\ f_-\bar s_0 \end{pmatrix} &=  \begin{pmatrix} (\del_{z}f_+ + f_+\del_{z}\log|s_0|^2)s_0  \\ (\del_z f_-)\bar s_0\end{pmatrix}.
\end{align*}

\begin{lemma}
Clifford connection preserves Clifford multiplication, i.e.
\begin{equation}
\label{eq:clif_con_mult}
\nabla^{Cl}_X(Y\cdot\psi) = (\nabla^{LC}_X Y)\cdot\psi + Y\cdot(\nabla_X^{Cl}\psi),
\end{equation}
where $\nabla^{LC}$ is the Levi-Civita connection.
\end{lemma}
\begin{proof} Defining $\overline{(\psi_+,\psi_-)} = (\bar\psi_-,\bar\psi_+)$, we see from the definition that $\overline{\nabla^{Cl}_X\psi} = \nabla^{Cl}_{\bar X}\bar\psi$. Hence, by linearity it is sufficient to check~\eqref{eq:clif_con_mult} in two cases: $X=Y=\del_{z}$ and $X=\bar Y = \del_{z}$.

If $X=Y=\del_{z}$, then in local coordinates for $\psi = (f_+s_0,f_-\bar s_0)$ one has
\begin{equation}
\label{CLXY}
\nabla^{Cl}_X(Y\cdot\psi) = 
\nabla^{Cl}_{\del_z}\begin{pmatrix} 0 \\ -\frac{f_+}{|s_0|^2} \bar s_0 \end{pmatrix} =
\begin{pmatrix} 0 \\ -\del_z\left(\frac{f_+}{|s_0|^2}\right) \bar s_0 \end{pmatrix}
\end{equation}
and
\begin{align*}
(\nabla^{LC}_X Y)\cdot\psi &+ Y\cdot(\nabla_X^{Cl}\psi) =\\
 &=\del_{z}(\log|s_0|^{-4})\del_z\cdot\begin{pmatrix} f_+s_0 \\ f_-\bar s_0 \end{pmatrix} + \del_z\cdot \begin{pmatrix} (\del_{z}f_+ + f_+\del_{z}\log|s_0|^2)s_0  \\ (\del_z f_-)\bar s_0\end{pmatrix} = \\
&= \frac{1}{|s_0|^2}\begin{pmatrix} 0\\2f_+\del_{z}(\log|s_0|^2) - (\del_{z}f_+ + f_+\del_{z}\log|s_0|^2)\end{pmatrix}
\end{align*}
and it is easy to see that the expression on the right-hand side is the same as  in \eqref{CLXY}. 

If $X=\bar Y=\del_{z}$, then in local coordinates for $\psi = (f_+s_0,f_-\bar s_0)$ one has
$$
\nabla^{Cl}_X(Y\cdot\psi) = 
\nabla^{Cl}_{\del_z}\begin{pmatrix} \frac{f_-}{|s_0|^2} s_0  \\ 0 \end{pmatrix} =
\begin{pmatrix} \left(\del_z\left(\frac{f_-}{|s_0|^2}\right) +\frac{f_-}{|s_0|^2}\del_z\log |s_0|^2\right) s_0 \\ 0  \end{pmatrix} = \begin{pmatrix} \frac{\del_z f_-}{|s_0|^2} s_0\\ 0\end{pmatrix}
$$
and 
\begin{align*}
(\nabla^{LC}_X Y)\cdot\psi &+ Y\cdot(\nabla_X^{Cl}\psi) = \del_{\bz} \cdot \begin{pmatrix} (\del_{z}f_+ + f_+\del_{z}\log|s_0|^2)s_0  \\ (\del_z f_-)\bar s_0\end{pmatrix}= \begin{pmatrix} \frac{\del_z f_-}{|s_0|^2} s_0\\ 0\end{pmatrix}.
\end{align*}
\end{proof}

Usually, the spinor bundle $\Sigma M$ is defined in terms of a $\Spin_n$-principal bundle $\Spin \, M$ over $M$ and a \emph{spinor representation} $\rho : \Spin_n \to \Sigma$ of $\Spin_n$ in the vector space $\Sigma$. The bundle $\Spin \, M$  correspond to a choice of a non-trivial 2-fold cover of $\mathrm{SO}_g M$, the $\mathrm{SO}_n$-principal bundle of positively oriented orthonormal bases of $TM$ with respect to the metric $g$. Then $\Sigma M$ is defined as the associated vector bundle to $(\Spin \, M, \rho):$ $\Sigma M = \Spin \, M \times_\rho \Sigma$. In the case $M$ is a Kähler manifold, the spinor bundle takes a special form, which in dimension 2 gives the isomorphism $\Sigma M \cong S \oplus \bar{S}$ (see \cite[Section 3.4]{F} for details). This justifies our definition of the spinor bundle as $S \oplus \bar{S}$. 

    The Dirac operator is normally defined in any dimension as
    \begin{align*}
        \D_g = \sum_{j = 1}^n e_j \cdot \nabla^{Cl}_{e_j},
    \end{align*}
    where $(e_1,\dots, e_n)$ is a local orthonormal basis of the tangent space. 
    In dimension 2, we compute
    \begin{align*}
        \D_g = e_1\cdot \nabla^{Cl}_{e_1} + e_2\cdot \nabla^{Cl}_{e_2} =2\left( |s_0|^2\del_{\bz}\cdot\nabla^{Cl}_{|s_0|^2\del_z} + |s_0|^2\del_{z}\cdot\nabla^{Cl}_{|s_0|^2\del_\bz} \right),
    \end{align*}
    so that 
    \begin{align*}
    \D_g\begin{pmatrix} f_+s_0 \\ f_-\bar s_0\end{pmatrix} = 2|s_0|^2\begin{pmatrix} (\del_zf_-)s_0\\ -(\del_\bz f_+)\bar s_0\end{pmatrix}, \end{align*} which coincides with  the definition we have been using earlier.

\subsection{Global criticality condition}

In this section, we need to change the way we enumerate the Dirac eigenvalues. Indeed if we enumerate the eigenvalues with $\lambda_1$ being the first positive eigenvalue, then any change in the dimension of the kernel of $\D$, which may happen when the conformal class changes \cite{Hit}, would cause a discontinuity of the eigenvalues. To have continuous eigenvalues $\lambda_k$, we instead index them by enumerating their squares in an increasing order:
\begin{align*}
    0 \leq \lambda_{\bi{1}}^2 \leq \lambda_{\bi{2}}^2 \leq \lambda_{\bi{3}}^2 \leq  \dots \nearrow +\infty.
\end{align*}
Here we use a bar over the indices to distinguish  this enumeration from the one we  used previously. In particular, note that $\lambda_{\bi{1}}$ is no longer the first positive eigenvalue and will be 0 if the kernel of $\D$ is non-empty
(this observation leads, in particular, to a positivity assumption in Proposition~\ref{prop:globcrit}).
Due to the symmetry of the Dirac spectrum with respect to 0 in dimension 2, we don't get any new eigenvalues when considering $\lambda^2$ instead of the non-negative eigenvalues, i.e. for any $\bi{k}$ for which $\lambda_{\bi{k}} > 0$, there exists a $k \geq 1$ such that $\lambda_{\bi{k}} = \lambda_k$ and the multiplicity of $\lambda_{\bi{k}}$ is twice the one of $\lambda_k$.

A formula for the variation of the Dirac eigenvalues under metric perturbations was obtained in  \cite{BG}.  It was  shown  that under analytic  variations $g(t)$ of the metric $g$, the derivative of an analytic branch $\lambda^{(j)}(t)$ of an eigenvalue $\lambda$ is given by
\begin{align} \label{eq:lambda-derivative}
    \frac{d}{dt} \lambda^{(j)}(t) \Big|_{t = 0} = - \frac{1}{2} \int_M \langle \dot{g}(0), Q_\psi \rangle dv_{g(0)},
\end{align}
where  the scalar product under the integral is understood as the scalar product of tensors induced by the metric $g$,  $\psi$ is a $L^2(g)$-normalised eigenspinor associated to the chosen branch $\lambda^{(j)}$ of $\lambda$, and
\begin{align*}
    Q_\psi(X,Y) = \frac{1}{2} \mathrm{Re} ( \langle X \cdot \nabla^{Cl}_Y \psi, \psi \rangle + \langle Y \cdot \nabla^{Cl}_X \psi, \psi \rangle).
\end{align*}
The bilinear form $Q_\psi$ is called the {\it energy-momentum tensor}, see  \cite{H,M} for other applications and its use in lower bounds for the first positive Dirac eigenvalue. 

Before we can state the definition of a critical metric, we need a way to specify the choice of the spin structure on $M$ which does not rely on the metric. Previously, when studying the conformally critical case, the spin structure $S$ could be defined solely in terms of the complex structure. But now that the complex structure, that is the conformal class, is not fixed, we will use an alternative definition of the spin structure.
\begin{definition}
    Let $\widetilde{\mathrm{Gl}}_2^+(\mathbb{R})$ be the universal cover of the linear positive group $\mathrm{Gl}_2^+(\mathbb{R})$.
    A \emph{classical spin structure} $\mathcal{S}$ is a $\widetilde{\mathrm{Gl}}_2(\mathbb{R})$-principal bundle which is a 2-fold cover of the $\mathrm{GL}_2^+(\mathbb{R})$-principal bundle $\mathcal{F}^+ M$ of positively oriented frames of $TM$. 
\end{definition}
\begin{remark}
The notion of the classical spin structure was  used in \cite{BG} and is essentially equivalent to the definition we have been using previously.   Given a metric $g$, a classical spin structure yields  a bundle $\Spin \, M \to \mathrm{SO}_g M$, by restricting $\mathcal{S} \to \mathcal{F}^+M$ to the bundle $\Spin \, M \subset \mathcal{S}$ covering $\mathrm{SO}_g \subset \mathcal{F}^+M$. In our notation, this means that a classical spin structure $\mathcal{S}$ gives a consistent way to associate to any metric $g$ a spin structure $S$. 
\end{remark}
Similarly to the conformally critical metric, we define a {\em globally} critical metric by the condition that the left and the right derivatives of the eigenvalue have opposite signs.
\begin{definition}
    Let $M$ be an oriented surface endowed with a classical spin structure $\mathcal{S}$. We say that the metric $g$ is \emph{$\bar{\lambda}_{\bi{k}}$-critical} if for any analytic family of smooth metrics $g(t), g(0) = g, t \in (-\epsilon, \epsilon)$, one has
    \begin{align*}
        \left(\frac{d}{dt} \bar\lambda_{\bi{k}}(M, g(t), S(t)) \Big|_{t = 0^-} \right) \left(\frac{d}{dt} \bar\lambda_{\bi{k}}(M, g(t), S(t))\Big|_{t = 0^+}\right) \leq 0,
    \end{align*}
    where $S(t)$ is a spin structure associated to $g(t)$ via the classical  spin structure $\mathcal{S}$.
\end{definition}
 Similarly to Remark~\ref{rmk:critBG}, our definition of critical metrics agrees with that in~\cite{BG} for simple eigenvalues, whereas it is more general for multiple eigenvalues. The following result is an extension of the criticality condition derived in \cite[p. 596]{BG}.

\begin{proposition} \label{prop:crit}
    Let $M$ be a compact oriented surface with spin structure $S$. Suppose that $g$ is a critical metric for $\bar{\lambda}_{\bi{k}}(g,S)$. Then there exist eigenspinors $\psi_1, \dots, \psi_m \in E_{\bi{k}}(g,S_g)$ such that
    \begin{align}
    \label{eq:crit_glob}
        \sum_{j = 1}^m Q_{\psi_j} = \frac{1}{2}\lambda_{\bi{k}}(g,S) g.
    \end{align}
    Conversely, if $\psi_1, \dots, \psi_m \in E_{\bi{k}}(g)$ are eigenspinors satisfying $\sum_{j = 1}^m Q_{\psi_j} = \frac{1}{2}\lambda_{\bi{k}}(g,S) g$ and $\lambda_{\bi{k}} < \lambda_{\bi{k+1}}$ or $\lambda_{\bi{k}} > \lambda_{\bi{k-1}}$ then $(g,S)$ is critical  for $\bar\lambda_{\bi{k}}$.
\end{proposition}

\begin{proof} The proof is analogous to the proof of Proposition~\ref{prop:conf-crit}.
    If $g$ is a critical metric, then  formula \eqref{eq:lambda-derivative}  yields that for any symmetric 2-tensor field $\eta$ with $\int_M \langle \eta, g \rangle dv_g = 0$, there exists $\psi \in E_{\bi{k}}(g, S)$, such that $\int_M \langle \eta, Q_\psi \rangle dv_g = 0$. 

    Let $W$ be the convex hull of $\{Q_\psi, \psi \in E_{\bi{k}}(g,S)\}$ in the space of bilinear forms. Then as before, we argue to show that $\frac{1}{2}\lambda_{\bi{k}} g \in W$. If not, by the Hahn--Banach theorem, there exists a symmetric 2-tensor field $\eta$ such that 
    \begin{align*}
        &\int_M \langle \eta, \lambda_{\bi{k}} g \rangle dv_g > 0 \\
        &\int_M \langle \eta, Q_\psi \rangle dv_g \leq 0, \qquad \text{for all } \psi \in E_{\bi{k}}(g,S).
    \end{align*}
    Setting $\eta_0 = \eta - \frac{1}{2 \vol(M)} \int_M \langle \eta, g \rangle dv_g g$, we have $\int_M \langle \eta_0, g \rangle dv_g = 0$,  and thus there exists $\psi \in E_{\bi{k}}(g,S)$ such that 
    \begin{align*}
        0 &= \int_M \langle \eta_0, Q_\psi \rangle dv_g \\
          &= \int_M \langle \eta, Q_\psi \rangle dv_g - \frac{1}{2 \vol(M)} \int_{M} \langle \eta, g \rangle dv_g \int_M \langle g, Q_\psi \rangle dv_g
          < 0 
    \end{align*}
    where we used that the trace of $Q_\psi$ is $\lambda_{\bi{k}} |\psi|^2$ since $\psi$ is an eigenspinor. This contradiction concludes the proof in one direction. The proof of the converse statement is also similar to the one in Proposition ~\ref{prop:conf-crit} and is left to the reader. 
\end{proof}

\subsection{Critical metrics and minimal surfaces}

In this section we give a geometric interpretation of the condition~\eqref{eq:crit_glob}, connecting the optimisation of Dirac eigenvalues to minimal surfaces in projective spaces. 

\begin{theorem} 
\label{thm:extful}
    Let $(M,g)$ be a compact oriented surface with spin structure $S$.  Suppose that $\psi_1, \ldots, \psi_m$ are eigenspinors on $M$ with $\D_g \psi_j = \lambda \psi_j, \lambda \neq 0$ and $\sum\limits_{j = 1}^m Q_{\psi_j} = \frac{1}{2}\lambda g$. Then the map $\Psi \colon M \to \CP^{2m-1}$ given in homogeneous coordinates by $[\psi_{1+} : \bar\psi_{1-} : \ldots : \psi_{m+} : \bar\psi_{m-}]$  is a branched minimal immersion.
\end{theorem}

\begin{proof}[Proof of Theorem \ref{thm:extful}]
    Let $e_1,e_2$ be a local orthonormal basis of $TM$. Then 
    \begin{multline*}
    \lambda_k = \sum_{j=1}^m(Q_{\psi_j}(e_1,e_1) +  Q_{\psi_j}(e_2,e_2)) = \\
    \frac{1}{2}\sum_{j=1}^m\mathrm{Re}(\la\D_g\psi_j,\psi_j\ra +\la\psi_j,\D_g\psi_j\ra ) = \lambda_k\sum_{j=1}^m|\psi_j|^2.
    \end{multline*}
   Hence, by Theorem~\ref{prop:harmo}, the map $\Psi$ is harmonic.
    
    For the conformality of $\Psi$, writing $\del_z = \frac{1}{2} (e_1 - i e_2)$, we have
    \begin{align*}
        0 &= g(e_1, e_1) - g(e_2,e_2) - 2ig(e_1,e_2) \\
          &= \sum_{j = 1}^m Q_{\psi_j}(e_1,e_1) - Q_{\psi_j}(e_2,e_2) - 2iQ_{\psi_j}(e_1,e_2) \\
          &= - \frac{1}{2} \sum_{j = 1}^m \big( \langle \nabla_{e_1} \psi_j, e_1 \cdot \psi_j \rangle + \langle e_1 \cdot \psi_j, \nabla_{e_1} \psi_j \rangle\big) 
          + \frac{1}{2}\sum_{j = 1}^m\big(\langle \nabla_{e_2} \psi_j, e_2 \cdot \psi_j \rangle + \langle e_2 \cdot \psi_j, \nabla_{e_2} \psi_j \rangle \big) \\
          & \quad + \frac{i}{2}\sum_{j = 1}^m \big( \langle \nabla_{e_1} \psi_j, e_2 \cdot \psi_j \rangle + \langle \nabla_{e_2} \psi_j, e_1 \cdot \psi_j \rangle 
           + \langle e_2 \cdot \psi_j, \nabla_{e_1} \psi_j \rangle + \langle e_1 \cdot \psi_j, \nabla_{e_2} \psi_j \rangle \big)  \\
          &= \sum_{j = 1}^m- \langle \nabla_{\del_z} \psi_j, e_1 \cdot \psi_j \rangle + i\langle  \nabla_{\del_z} \psi_j, e_2 \cdot \psi_j \rangle - \langle e_1 \cdot \psi_j, \nabla_{\del_\bz} \psi_j \rangle
          +i \langle e_2 \cdot \psi_j, \nabla_{\del_\bz} \psi_j \rangle \\
          &= - 2 \sum_{j = 1}^m  \big(\langle \nabla_{\del_z} \psi_j, \del_\bz \cdot \psi_j \rangle + \langle \del_z \cdot \psi_j, \nabla_{\del_\bz} \psi_j \rangle \big). 
    \end{align*}

    Writing $\psi_j = (f_{j+} s_0, \bar{f}_{j-} \bar{s}_0)$, and using that $\del_\bz \cdot \begin{pmatrix}\psi_+\\ \psi_-\end{pmatrix} = |s_0|^2 \begin{pmatrix} 0 & 1 \\ 0 & 0 \end{pmatrix}\begin{pmatrix}\psi_+\\ \psi_-\end{pmatrix}$ and $\del_z \cdot \begin{pmatrix}\psi_+\\ \psi_-\end{pmatrix}= |s_0|^2 \begin{pmatrix} 0 & 0 \\ -1 & 0 \end{pmatrix}\begin{pmatrix}\psi_+\\ \psi_-\end{pmatrix}$ in this basis, 
    \begin{align*}
 0 = \sum_{j = 1}^m  \del_z f_{j+}\overline{(-\bar{f}_{j-})} + f_{j+} \overline{\del_\bz \bar{f}_{j-}} = \sum_{j = 1}^m - f_{j-} \del_z f_{j+} + f_{j+} \del_z f_{j-}.
    \end{align*}
    Writing $F = (f_{1+}, f_{1-}, \dots, f_{m+}, f_{m-})$ and $I(a_{1+}, a_{1-}, \dots) = (-\bar{a}_{1-}, \bar{a}_{1+}, \dots)$, we just proved $\langle \del_z F, I(F) \rangle_{\CC^{2m}}=0$.
    By the eigenvalue equation~\eqref{eq:eigsp_qharm}, $I(F) = \frac{2|s_0|^2}{\lambda} \del_\bz F$, therefore,
\begin{equation}
\label{eq:weak_conf}
\langle \del_z F, \del_\bz F \rangle_{\CC^{2m}}=0.
\end{equation}
At the same time, a harmonic map $\Psi$ is a branched minimal immersion if and only if it is weakly conformal \cite{GOR73}, i.e.  $\la d_\C\Psi(\del_z),d_\C\Psi(\del_{\bar z})\ra = 0$, or, equivalently,  $\la \del \Psi(\del_z),\bar\del\Psi(\del_{\bar z})\ra = 0$. Using the definition of the metric on $\CP^{2m-1}$, one has
$$
\la \del \Psi(\del_z),\bar\del\Psi(\del_{\bar z})\ra = 4|s_0|^4\frac{\la\pi_{L^\perp}\del_zF, \pi_{L^\perp}\del_{\bar z}F\ra}{|F|^2} = 4|s_0|^4\frac{\la\del_zF, \pi_{L^\perp}\del_{\bar z}F\ra}{|F|^2},
$$
where $\pi_{L^\perp}$ is an orthogonal projection to the orthogonal complement to the one-dimensional space spanned by $F$. Finally, the eigenvalue equation~\eqref{eq:eigsp_qharm} implies that $\pi_{L^\perp}\del_{\bar z}F =\del_{\bar z}F$, so that~\eqref{eq:weak_conf} yields that $\Psi$ is weakly conformal.   
\end{proof}

We conclude the article by stating the analogue of Proposition~\eqref{prop:spin-harmonic} for a globally critical metric.

\begin{definition}
We say that a map $\Psi\colon (M,\mC)\to\CP^{2m-1}$ is a {\em quaternionic branched minimal immersion} if $\Psi$ is a quaternionic harmonic map and a branched minimal immersion.
\end{definition}
\begin{proposition}
\label{prop:globcrit}
    Let $M$ be an oriented surface and $\mathcal{S}$ be a classical spin structure. Suppose that $g$ is $\bar\lambda_{\bi{k}}$-critical, and $\bar\lambda_{\bi{k}}(M,g,S)>0$.
 Then there exists a collection $\psi_j = (\psi_{j+},\psi_{j-})$, $j=1,\ldots, m$ of $\lambda_{\bi{k}}$-eigenspinors such that the map $\Psi\colon(M,\mC)\to\CP^{2m-1}$ given in homogeneous coordinates by
$$	
\Psi = [\psi_{1+} : \bar\psi_{1-}: \dots : \psi_{m+}: \bar\psi_{m-}]
$$	
is a quaternionic branched minimal immersion. Moreover, $\bar\lambda_{\bi{k}}(M,g,S)^2 = E^{(0,1)}_g(\Psi)$ and $g = \alpha g_\Psi = \alpha |\bar\del\Psi|_g^2g$ for some $\alpha>0$.	

Conversely, let $\Psi\colon (M,\mC)\to\CP^{2m-1}$ be a quaternionic branched minimal immersion,  and let $S_\Psi$ be the spin structure induced by $\Psi$ in the sense of Proposition~\ref{prop:H_harm}. Then the metric $g_\Psi$ is $\lambda_{\bi{k}}$-critical for the classical spin structure $\mathcal{S}_\Psi$ corresponding to $S_\Psi$. The index $\bi{k}-1$ is the number of eigenvalues $\lambda_{\bi{j}}(M,g_\Psi,S_\Psi)$ satisfying $0<\lambda_{\bi{j}}(M,g_\Psi,S_\Psi)<1$. Furthermore, one has $\bar\lambda_{\bi{k}}(M,g_\Psi,S_\Psi)^2 = E^{(0,1)}_\mC(\Psi)>0$.
\end{proposition}
\begin{proof}
To prove the first part, observe that we showed at the beginning of Section~\ref{sec:eigmaps} that the maps by eigenspinors satisfying $\sum|\psi_j|^2 = 1$ are quaternionic harmonic, and  hence $\Psi$ is quaternionic harmonic. Furthermore, it follows from Theorem~\ref{thm:extful} that $\Psi$ is a branched minimal immersion. The proofs of the statements about $g$ and the value of $\bar\lambda_{\bi{k}}(M,g,S)$ are analogous to the first part of Proposition~\ref{prop:spin-harmonic}.

To prove the converse, observe that if $\Psi$ is a quaternionic branched minimal immersion, then 
according to Proposition~\ref{prop:H_harm} it is the map by $g_\Psi$-eigenspinors satisfying $\sum|\psi_j|_{g_\Psi}^2 = 1$. Tracing the computations in the proof of Theorem~\ref{thm:extful} in the opposite direction, we observe that weak conformality of $\Psi$ implies 
$$
\sum_{j=1}^m Q_{\psi_j}(e_1,e_1) = \sum_{j=1}^m Q_{\psi_j}(e_2,e_2),\qquad \sum_{j=1}^m Q_{\psi_j}(e_1,e_2) = 0,
$$ 
where $e_1,e_2$ is $g_\Psi$-orthonormal basis. At the same time, by the eigenspinor equation
\begin{equation*}
\begin{split}
&\sum_{j=1}^m(Q_{\psi_j}(e_1,e_1) +  Q_{\psi_j}(e_2,e_2)) =\\
    &\frac{1}{2}\sum_{j=1}^m\mathrm{Re}(\la\D_g\psi_j,\psi_j\ra +\la\psi_j,\D_g\psi_j\ra ) = \lambda_{\bi{k}}\sum_{j=1}^m|\psi_j|_{g_\Psi}^2 = \lambda_{\bi{k}}.
 \end{split}
\end{equation*}
Thus,
$$
\sum_{j=1}^m Q_{\psi_j} = \frac{\lambda_{\bi{k}}}{2}g_{\Psi},
$$
which by Proposition~\ref{prop:crit} means that $(g_\Psi,S_\Psi)$ is critical for $\bar\lambda_{\bi{k}}$.
\end{proof}

\begin{remark}
It is proved in~\cite{GG} that if $M$ is not a  sphere,
$$
\inf_{\mC}\bar\lambda_{\bi{1}}(M,\mC,S) = 0,
$$
and hence Proposition \ref{prop:globcrit} does not yield any global $\bar\lambda_{\bi{1}}$-minimisers.
It would be interesting to find out if Proposition~\ref{prop:globcrit} can be used to construct examples of quaternionic branched minimal immersions of surfaces of positive genus. It is likely  that the proof of~\cite{GG} can be generalised to higher eigenvalues. If true, it means that in order to obtain quaternionic branched minimal immersions from Proposition~\ref{prop:globcrit}, one would need to look for saddle points rather than global minima. 
\end{remark}

\subsection*{Acknowledgments} Part of this work was completed while A.M. was a Ph.D. student at the Universit\'e de Montr\'eal, and was  supported by FRQNT grant B2-272578. He was also  supported  by  EPSRC grant 
EP/T030577/1. The research of M.K. was partially supported by the NSF grant DMS-2104254. The research of I.P. was partially supported by NSERC and FRQNT.

\end{document}